\documentclass[a4paper, 10pt, american]{amsart}
\usepackage{times,latexsym,amssymb}
\usepackage{amsmath,amsthm,bm}
\usepackage{color}
\usepackage[colorlinks,pdfpagelabels,pdfstartview = FitH,bookmarksopen
= true,bookmarksnumbered = true,linkcolor = blue,plainpages =
false,hypertexnames = false,citecolor = red,pagebackref=false]{hyperref}

\usepackage{amsbsy}
\usepackage{amstext}
\usepackage{amssymb}
\usepackage{esint}
\usepackage{stmaryrd}
\usepackage[pdftex]{graphicx}
\usepackage{floatflt}

\allowdisplaybreaks
\newtheorem{theorem}{Theorem}
\newtheorem{lemma}[theorem]{Lemma}
\newtheorem{definition}[theorem]{Definition}

\newtheorem{remark}[theorem]{Remark}
\numberwithin{theorem}{section}
\numberwithin{equation}{section}

\newcommand{\mint}{- \mskip-19,5mu \int}

\def\N{\mathbb{N}}
\def\R{\mathbb{R}}

\newcommand{\A}[0]{\ensuremath{\mathbf{A}(x,t,u,D\power um)}}

\newcommand{\dxdt}{\,\mathrm{d}x\mathrm{d}t}

\renewcommand{\d}{\,\mathrm{d}}
\newcommand{\dx}{\,\mathrm{d}x}

\newcommand{\dt}{\,\mathrm{d}t}

\renewcommand{\epsilon}{\varepsilon}

\DeclareMathOperator{\Div}{div}

\renewcommand{\epsilon}{\varepsilon}
\newcommand{\eps}{\varepsilon}
\renewcommand{\rho}{\varrho}

\def\eqn#1$$#2$${\begin{equation}\label#1#2\end{equation}}

 \newcommand{\power}[2]{\bm{#1^{\mbox{\unboldmath{\scriptsize$#2$}}}}}
 \newcommand{\abs}[1]{|#1|}

\def\Xint#1{\mathchoice
    {\XXint\displaystyle\textstyle{#1}}%
    {\XXint\textstyle\scriptstyle{#1}}%
    {\XXint\scriptstyle\scriptscriptstyle{#1}}%
    {\XXint\scriptscriptstyle\scriptscriptstyle{#1}}%
    \!\int}
\def\XXint#1#2#3{\setbox0=\hbox{$#1{#2#3}{\int}$}
    \vcenter{\hbox{$#2#3$}}\kern-0.5\wd0}
\def\bint{\Xint-}
\def\dashint{\Xint{\raise4pt\hbox to7pt{\hrulefill}}}

\def\Xiint#1{\mathchoice
    {\XXiint\displaystyle\textstyle{#1}}%
    {\XXiint\textstyle\scriptstyle{#1}}%
    {\XXiint\scriptstyle\scriptscriptstyle{#1}}%
    {\XXiint\scriptscriptstyle\scriptscriptstyle{#1}}%
    \!\iint}
\def\XXiint#1#2#3{\setbox0=\hbox{$#1{#2#3}{\iint}$}
    \vcenter{\hbox{$#2#3$}}\kern-0.5\wd0}
\def\biint{\Xiint{-\!-}}

\def\<{\langle}
\def\>{\rangle}
\def\Tilde{\widetilde}

\subjclass[2010]{35B65, 35K67, 35K86, 47J20}
\keywords{Porous medium equation, obstacle problem, higher integrability}

\begin{document}
\title[Higher integrability in the obstacle problem for the fast diffusion equation]{Higher integrability in the obstacle problem for the\\ fast diffusion equation}
\date{\today}

\author[Y. Cho]{Yumi Cho}
\address{Yumi Cho\\
Department of Mathematical Sciences, Seoul National University\\
08826 Seoul, Republic of Korea}
\email{imuy31@snu.ac.kr}

\author[C. Scheven]{Christoph Scheven}
\address{Christoph Scheven\\ Fakult\"at f\"ur Mathematik, 
Universit\"at Duisburg-Essen\\45117 Essen, Germany}
\email{christoph.scheven@uni-due.de}

\thanks{Y. Cho was supported by NRF-2019R1I1A1A01064053.}

\maketitle

%****************************************************************

\begin{abstract}
 We prove local higher integrability of the spatial
 gradient for solutions to obstacle problems of porous medium type in
 the fast diffusion case $m<1$. The result holds for the natural range
 of exponents that is known from other regularity results for
 porous medium type equations. We also
 cover the case of signed solutions. 
\end{abstract}

\section{Introduction}
In this paper we consider the obstacle problem associated with equations
of porous medium type, whose prototype is given by 
\begin{equation}\label{FDE}
  \partial_t u - \Delta\big(|u|^{m-1}u\big) = 0
  \quad\mbox{in }\Omega_T:=\Omega\times(0,T),
\end{equation}
with possibly signed solutions $u:\Omega_T\to\R$.
Here, we study the singular case $\frac{(n-2)_+}{n+2}<m<1$,
in which \eqref{FDE} is known as \emph{singular porous medium
  equation} or \emph{fast diffusion equation}.
In this case, the behaviour of solutions to~\eqref{FDE} differs
considerably from the behaviour in the case $m>1$.
For example, in the case $m<1$
solutions exhibit infinite propagation speed and we may observe extinction in finite time.
This is in contrast to the case $m>1$ of slow diffusion, in which 
disturbances propagate with finite speed and the solution
might vanish outside of a compact subset of the spatial domain.

Here, we consider a large class of equations that generalize the fast diffusion
equation~\eqref{FDE}. 
More precisely, we deal
with obstacle problems related to equations of the kind
\begin{equation}\label{por-med-eq}
  \partial_t u -\Div \mathbf A\big(x,t,u,D(|u|^{m-1}u)\big) =g-\Div F
  \quad\mbox{in $\Omega_T$}, 
\end{equation}
where the vector field $\mathbf
A\big(x,t,u,\xi):\Omega_T\times\R\times\R^n\to\R^n$ satisfies ellipticity and
growth conditions comparable to the model case $A(\xi)=\xi$, which
corresponds to~\eqref{FDE}.
For the obstacle problem, we additionally consider a constraint of the
form $u\ge\psi$ a.e. in $\Omega_T$ for a given obstacle function
$\psi:\Omega_T\to\R$, and replace \eqref{por-med-eq} by an associated
variational inequality. 
For the precise assumptions we refer to 
Section \ref{sec:result} below.
Our notion of solution includes the regularity property
$D(|u|^{m-1}u)\in L^2(\Omega_T,\R^n)$, which is natural in view of
existence results, cf. \cite{BLS:MathAnn05,AltLuckhaus}. The objective of the
present paper is to establish the self-improving property of higher
integrability for the spatial gradient, i.e. to prove that every
solution to the obstacle problem satisfies
$D(|u|^{m-1}u)\in L^{2+\eps}_{\mathrm{loc}}(\Omega_T,\R^n)$ for some $\eps>0$.

The question for higher integrability of solutions to partial
differential equations has a long history.
In the case of elliptic systems, the first result of this type goes
back to the work \cite{Meyers-Elcrat} by  Elcrat \& Meyers, which in
turn is based on the
ground-breaking work of Gehring \cite{Gehring}. For further results in
the elliptic setting 
and their applications to regularity theory, we refer to the monographs 
\cite[Chapter~V, Theorem 2.1]{Giaquinta:book} and \cite[Section
6.4]{Giusti:book}. The case of elliptic obstacle problems has been
considered in \cite{Eleuteri:2007,EleuteriHabermann:2008}.
The first parabolic analogue was established by Giaquinta \&
Struwe \cite[Theorem 2.1]{Giaquinta-Struwe}, who settled the question
of higher integrability in the case of 
parabolic systems with quadratic growth. For the more general case of
parabolic systems of
$p$-Laplace type with $p\neq2$, however, the solution of the
corresponding problem turned out to be much more involved and stayed
open for some time. The additional issue in this case is caused by the
anisotropic scaling behaviour of the parabolic $p$-Laplace operator,
which leads to inhomogeneous estimates and therefore prevents the
application of the classical Gehring lemma. This problem was 
eventually overcome by Kinnunen \& Lewis
\cite{Kinnunen-Lewis:1}, who verified the property of higher integrability for the spatial gradient of
solutions to parabolic $p$-Laplace systems in the range $p>\frac{2n}{n+2}$. 
As a key tool to compensate the anisotropic nature of the equation,
they used the idea of intrinsic parabolic cylinders by
DiBenedetto. For an exposition of this idea of intrinsic geometry and further applications,
we refer to the monographs \cite{DiBe,DBGV-book}. The breakthrough by
Kinnunen \& Lewis initiated various related research activities,
cf. \cite{AdimurthiByun:2018, Boegelein:1, Boegelein:2,Boegelein-Duzaar:1,
  Boegelein-Parviainen, Kilpelainen-Koskela, Parviainen, Parviainen-singular}. In particular, the
case of obstacle problems of parabolic $p$-Laplace type has been
treated in \cite{Boegelein-Scheven-obs}.  

For the porous medium equation, the problem of higher integrability of
the spatial gradient poses even more challenges than in the case of
the parabolic $p$-Laplace equation. This stems from the fact that the
modulus of ellipticity of the porous medium operator degenerates when
$|u|$ becomes small, while the degeneracy of the $p$-Laplacian depends
on the size of the gradient $|Du|$.
The type of degeneracy of the porous medium equation 
makes it considerably harder to prove gradient estimates. On a
technical level, the problem is to construct a system of intrinsic
cylinders that compensates the possible degeneracy in $|u|$ and at the
same time, takes into account the size of the spatial gradient. 
This major problem was eventually solved by Gianazza \& Schwarzacher
\cite{Gianazza-Schwarzacher} who established the self-improving
property of higher integrability for the spatial gradient of
non-negative solutions to the porous medium equation in the slow
diffusion range $m>1$. Their method relies on the phenomenon of expansion of
positivity and is therefore limited to the scalar case and to
non-negative solutions.
The extension of the result to signed solutions and to
systems of porous medium type 
is due to \cite{Bogelein-Duzaar-Korte-Scheven}.
In the fast diffusion case $m<1$, the different behaviour
of the singular porous medium equation requires non-trivial
adaptations of the techniques from \cite{Gianazza-Schwarzacher}. This was
carried out for non-negative solutions in
\cite{Gianazza-Schwarzacher:m<1} and for the vectorial case in
\cite{BDS-singular-higher-int}.
The techniques that have been developed in the mentioned works also
proved to be useful for several related problems. 
For example, in \cite{BDKS:2018,Saari-Schwarzacher}
the authors adapted these techniques
to prove higher integrability for a certain class of 
doubly nonlinear equations, which include both the degeneracy of the
porous medium equation and of the parabolic $p$-Laplacian.
For the degenerate porous medium equation, the
higher integrability was extended up to the boundary in \cite{MSSS}
and to obstacle problems in \cite{Cho-Scheven}. The present work is
devoted to the question whether the result of the latter work also
applies in the singular case of the fast diffusion equation.

After finalizing this paper we became aware that independently of us,
a partial answer to this question has been given in
\cite{QifanLi}. However, the approach in \cite{QifanLi} is limited to
solutions that are non-negative and bounded. Therefore, the author has
to assume very strong regularity properties of the obstacle such as
Lipschitz continuity and boundedness of the time derivative of a
certain power of the solution. On the contrary, we
are able to treat signed solutions that may be unbounded and can deal
with general obstacle functions in the natural energy space.

We conclude the introduction with some remarks on the strategy of the proof. As
mentioned above, the identification of a suitable intrinsic geometry
plays a crucial role in the arguments. Since our goal is the estimate of
the spatial gradient $D(|u|^{m-1}u)$, we consider a
parameter $\theta^m$ that is related to the quantity
$\frac{|u|^m}r$,which has the same scaling behaviour as
$|D(|u|^{m-1}u)|$. Then we work with cylinders of the type
\begin{equation}\label{ratio-cylinders}
  Q_{r,s}:=B_{r}(x_o)\times(t_o-s,t_o+s)
  \qquad\mbox{with}\qquad
  \frac{s}{r^{\frac{1+m}{m}}}=\theta^{1-m}.
\end{equation}
Heuristically, the last identity means that $s\approx|u|^{1-m}r^2$. 
This form of the cylinders reflects the fact that the modulus of
ellipticity of the porous medium equation is proportional to
$|u|^{m-1}$ and at the same time, the cylinders are adapted to the derivation of
gradient estimates. 
The heuristical idea that $\theta^m$ corresponds to the size of
$\frac{|u|^m}{r}$ is made precise by the requirement
\begin{equation}\label{intro:intrinsic-cylinder}
  \theta^{1+m}
  \le
  B\biint_{Q_{r,s}}\frac{|u|^{1+m}}{r^{\frac{1+m}{m}}}\dx\dt
  \le
  B^2\theta^{1+m},
\end{equation}
for a constant $B\ge1$. 
If $Q_{r,s}$ satisfies this condition, we say that the cylinder is
\emph{intrinsic}. However, it turns out that this property is too strong
to be achieved in every case. 
A decisive idea by Gianazza \& Schwarzacher \cite{Gianazza-Schwarzacher}
is therefore to relax this assumption and to work with \emph{sub-intrinsic
  cylinders},  which means that only the second inequality in
\eqref{intro:intrinsic-cylinder} holds true. In order to obtain the
missing bound of $\theta$
from above, we consider two cases: the \emph{singular regime}
that is characterized by the fact that the solution is small in terms
of its oscillation, more precisely
\begin{equation}\label{intro:singular}
  \theta^{2m}
  \le
  B\biint_{Q_{r,s}}\big(\big|D(|u|^{m-1}u)\big|^2+G\big)\dx,
\end{equation}
where $G\in L^{1+\eps}(\Omega_T)$ is a function that depends on the data, and the
\emph{non-singular regime}, in which the construction of an intrinsic
cylinder with~\eqref{intro:intrinsic-cylinder} is possible.

For both cases, we establish a reverse H\"older inequality on the
respective cylinders, which is the prerequisite for any higher
integrability result. The usual ingredients for the proof of the
reverse H\"older inequality are an energy estimate and a parabolic
Sobolev-Poincar\'e inequality, cf. Section~\ref{sec:energy-poincare}.
In particular the second one is very involved in the case of the
obstacle problem. In order to control the oscillation in time
direction in the absence of a weak time derivative, one needs to
control the difference of mean values over different time
slices. However, for a solution to the obstacle problem, it is difficult to
obtain any information on the set where the solution is close to the
obstacle. This problem leads to additional integrals over the set
$\{|u|^{m-1}u\le|\psi|^{m-1}\psi+\mu^m\}$, where $\mu>0$ is an
arbitrary parameter (cf. the proof of
Lemma~\ref{lem:time-diff}). These additional integrals have to be
dealt with by choosing the parameter $\mu$ carefully in order to
compensate the inhomogeneity of the estimates. By distinguishing
between the singular and the non-singular case, we thereby obtain the
desired Sobolev-Poincar\'e inequality, see
Lemma~\ref{lem:sobolev-poin}, and finally the reverse H\"older
inequality in Lemma~\ref{lem:reverseHolder}.

This inequality is then exploited on a suitable system of
sub-intrinsic cylinders that satisfy either
\eqref{intro:intrinsic-cylinder} or
\eqref{intro:intrinsic-cylinder}$_2$, \eqref{intro:singular}.
For the details of the construction, which originates from
\cite{Gianazza-Schwarzacher,BDS-singular-higher-int}, we refer to
Section~\ref{sec:hi}. By covering the super-level sets of the gradient
with these cylinders and applying the reverse H\"older inequality on
each of them, we can then derive the desired higher integrability
estimate along the same lines as in the case of the
parabolic $p$-Laplacian.  

Both for the derivation of the Sobolev-Poincar\'e inequality and for
the construction of the sub-intrinsic cylinders, we rely on the lower bound
\begin{equation*}
  m>\frac{(n-2)_+}{n+2}=:m_c.
\end{equation*}
This restriction of the exponent is natural and appears also in other
regularity results, see \cite[Section 6.21.6]{DBGV-book} for a
discussion of the critical exponent $m_c$. This exponent also occurs
naturally
as a borderline case in a parabolic Sobolev embedding that is related to
our setting, see~\eqref{integrability-Kpsi} below.

\section{Notation and the main result}

Before stating our main result, we introduce some notations to be used
throughout the paper.

\subsection{Notation}
As short-hand notation for the power of a signed value $a\in \R$ we define
$$\power ak := |a|^{k-1}a \quad\mbox{for $k\in(0,\infty)$}.$$
We use the following parabolic cylinders:
\begin{equation}\label{def-Q}
	Q_\rho^{(\theta)}(z_o)
	:=
	B_\rho^{(\theta)}(x_o)\times\Lambda_\rho(t_o)
\end{equation}
for $z_o=(x_o,t_o)\in\R^n\times(0,T)$, $n\in\N$, and for $\rho,\theta>0$,
where  
\begin{align*}
B_\rho^{(\theta)}(x_o):=B_{\rho\theta^\frac{m(m-1)}{1+m}}(x_o)
=\big\{x\in\R^n:|x-x_o|<\rho\theta^\frac{m(m-1)}{1+m}\big\}
\end{align*}
and
\begin{equation*}
  \Lambda_\rho(t_o)
  :=
  \big(t_o-\rho^{\frac{1+m}{m}},t_o+\rho^{\frac{1+m}{m}}\big).
\end{equation*}
We note that the cylinders $Q_\rho^{(\theta)}(z_o)$ are of the
type~\eqref{ratio-cylinders} mentioned in the introduction with
$r=\rho\theta^{\frac{m(m-1)}{1+m}} $ and $s=\rho^{\frac{1+m}m}$.
If $\theta=1$, we simply write
\begin{align*}
B_\rho(x_o):=B_\rho^{(1)}(x_o)
\quad\mbox{and}\quad
Q_\rho(z_o):=B_\rho^{(1)}(x_o)\times\Lambda_\rho(t_o).
\end{align*}
Moreover, we define the boundary term
\begin{align}\label{def:b}
	\mathfrak b[\power{u}{m},\power{v}{m}]
	:=
	\tfrac{m}{1+m}\big(\abs{v}^{1+m} - \abs{u}^{1+m}\big) - u\big(\power{v}{m} - \power{u}{m}\big)
\end{align}
for $u,v\in\R$. 
This has the property
\begin{align}\label{boundary-term}
\frac1{c}\big|\power{u}{\frac{1+m}{2}}-\power{v}{\frac{1+m}{2}}\big|^2 
	\le
	\mathfrak b[\power{u}{m},\power{v}{m}]
	\le
	c \big|\power{u}{\frac{1+m}{2}}-\power{v}{\frac{1+m}{2}}\big|^2
\end{align}
for some constant $c=c(m)>0$, see \cite[Lemma 3.4]{BDKS:2018}.
We also define the slicewise mean  $\<h\>_B(t)$ of $h$ over a
measurable set $B\subset\Omega\subset\R^n$ for a.e. $t\in(0,T)$ by
\begin{equation}\label{slice-mean}
	\<h\>_{B}(t)
	:=
	\bint_{B} h(t)\dx,
\end{equation}
while  the mean value $(h)_{D}\in \R$ of $h$ over a measurable set
$D\subset\Omega\times(0,T)$ in space-time is given by
\begin{equation*}
	(h)_{D}
	:=
	\biint_{D} h \dxdt,
\end{equation*}
provided $h\in L^1(0,T;L^1(\Omega,\R))$.
We introduce the mollification in time 
\begin{align}\label{time-molli}
\llbracket w \rrbracket_h(x,t):=\frac1h\int_0^t e^\frac{s-t}h w(x,s)\d s,
\quad h>0
\end{align}
for $w\in L^1(0,T;L^1(\Omega,\R))$ and a.e. $(x,t)\in\Omega_T$.

\subsection{Assumptions and statement of the result}
\label{sec:result}
On an arbitrary domain $\Omega\subset\R^n$, where $n\in\N$,
we deal with the obstacle problem related to
the singular porous medium equation with an exponent $m$
in the range
\begin{equation}
  \label{range-m}
  \frac{(n-2)_+}{n+2}<m\le1.
\end{equation}
Moreover, we consider an obstacle function $\psi : \Omega_T \rightarrow \R$ with
\begin{equation}\label{obs-func}
\power \psi m \in L^2(0,T; W^{1,2}(\Omega))\quad\mbox{and} \quad \partial_t \power \psi m \in L^{\frac{1+m}{m}}(\Omega_T),
\end{equation}
and inhomogeneities
\begin{align}\label{inhom-func}
g\in L^{1+m}(\Omega_T,\R)
\quad\mbox{and}\quad
F\in L^2(\Omega_T,\R^n).
\end{align}
Note that \eqref{obs-func} implies $\psi\in
C^0\big([0,T];L^{1+m}(\Omega)\big)$, since $m\ge\frac{n-2}{n+2}$.
We then define the class of functions
\begin{align*}
	K_{\psi}
	:=
	\big\{v \in C^0\big([0,T];L^{1+m}(\Omega)\big):
	\power v m\in L^2\big(0,T;W^{1,2}(\Omega)\big), \
	v\ge\psi\mbox{ a.e. in $\Omega_T$}\big\}
\end{align*}
and 
$$
	K_{\psi}'
	:=
	\big\{v\in K_{\psi}: \partial_t \power v m\in L^{\frac{1+m}{m}}(\Omega_T)\big\}.
$$
In addition, we assume that the vector field $\mathbf A=\mathbf
A(x,t,u,\xi)\colon \Omega_T\times \R\times\R^{n}\to\R^{n}$ is
measurable in $(x,t)\in\Omega_T$ for every $(u,\xi)\in\R\times\R^n$,
continuous in $(u,\xi)\in\R\times\R^n$ for a.e. $(x,t)\in\Omega_T$, and
satisfies the following growth and ellipticity conditions with some constants $0<\nu\le L<\infty$:
\begin{equation}\label{A-struc}
\left\{ \begin{array}{l}
|\mathbf{A}(x,t,u,\xi)| \leq L |\xi|,\\[0.6ex]
\mathbf{A}(x,t,u,\xi)\cdot \xi \geq \nu |\xi|^2,
\end{array}\right.\end{equation}
for a.e. $(x,t)\in\Omega_T$ and for every $(u,\xi)\in\R\times\R^n$.
Under the above assumptions, we consider local weak solutions in the
following sense.

\begin{definition}
  We say that a function $u\in K_\psi$ is a local weak solution of the
  obstacle problem related to the equation \eqref{por-med-eq} if it
  satisfies the variational inequality
  \begin{align}\label{vari-ineq}
    &\langle\!\langle\partial_tu,\alpha\eta(\power vm-\power um)\rangle\!\rangle
    +
    \iint_{\Omega_T}\alpha \A\cdot D\big(\eta(\power vm-\power um)\big)\dxdt \\\nonumber
    &\hspace{2.7cm}
    \ge \iint_{\Omega_T} \alpha \big[\eta g (\power vm-\power um)+F \cdot D\big(\eta(\power
    vm-\power um)\big) \big] \dxdt
  \end{align}
  for every comparison function $v\in K'_{\psi}$\,, any cut-off
  function $\alpha\in W^{1,\infty}_0([0,T],\R_{\ge0})$ in time, and
  any cut-off function $\eta\in W^{1,\infty}_0(\Omega,\R_{\ge0})$ in
  space, where
  \begin{equation}\label{weak-time-term}
    \langle\!\langle\partial_t u,\alpha\eta (\power vm-\power um)\rangle\!\rangle
    :=
    \iint_{\Omega_T} \eta \big\{ \alpha'\big(\tfrac1{1+m}|u|^{1+m}-u\power{v}{m}\big)
    -\alpha u\partial_t \power{v}{m} \big\}\dxdt.
  \end{equation}
\end{definition}

An existence result for solutions in this sense has been established
in \cite{BLS:MathAnn05}.

We now present our main result.
\begin{theorem}\label{thm:main}
Let $\frac{(n-2)_+}{n+2}<m\le1$ and let $u\in K_\psi$ be a local weak
solution to the variational inequality \eqref{vari-ineq}.
Then there exist constants $\tau_o=\tau_o(n,m,\nu,L)\in(0,1]$ and $c=c(n,m,\nu,L)\ge1$ such that if
\begin{align}\label{def-G}
G:=|F|^2+|g|^{1+m}+|D\power\psi m|^2 + |\partial_t\power \psi m|^\frac{1+m}m \in L^{1+\frac\gamma2}_{\rm loc}(\Omega_T)
\end{align}
for some $\gamma>0$,
then we have
\begin{align*}
|D\power um|\in L^{2+\tau_1}_{\rm loc}(\Omega_T) 
\quad\mbox{with $\tau_1:=\min\{\tau_o,\gamma, \frac{4m}{1+m}+\frac4n-2\}$}.
\end{align*}
Moreover, we have the estimate
\begin{align}\label{desired-est-1}
&\biint_{Q_{R}(z)}|D\power um|^{2+\tau}\dxdt\nonumber\\
&\qquad\le
c\Bigg[ \biint_{Q_{2R}(z)} \bigg(\frac{|u|^{1+m}}{R^\frac{1+m}m}+G+1\bigg)\dxdt\Bigg]^\frac{d\tau }{2} \biint_{Q_{2R}(z)}|D\power um|^2\dxdt\nonumber\\
&\qquad\qquad+
c\,\biint_{Q_{2R}(z)}(G+|u|^{1+m})^{1+\frac\tau2}\dxdt
\end{align}
for every $\tau\in(0,\tau_1]$ and for every cylinder $Q_{2R}(z)\subset\Omega_T$ with $0<R\le1$,
where the scaling deficit $d$ is given by
\begin{align}\label{def-d}
d:=\frac{2(1+m)}{2(1+m)-n(1-m)}.
\end{align}
 \end{theorem}

\begin{remark}\upshape
 The appearance of the scaling deficit $d$ in the above estimate is a
 natural consequence of the anisotropic scaling behaviour of the
 porous medium equation. We note that $d=1$ in the case $m=1$, which
 corresponds to the linear case of the heat equation. Moreover, the
 scaling deficit tends to infinity if $m$ approaches the critical
 exponent $\frac{n-2}{n+2}$.  
\end{remark}

\section{Auxiliary material}

\subsection{An iteration lemma}
We start with an elementary lemma that follows from \cite[Lemma~6.1]{Giusti:book} by a change of variables.

\begin{lemma}\label{lem:tech}
Let $0<\vartheta<1$, $A,C\ge 0$ and $\alpha,\beta > 0$. Then, there exists a constant  $c = c(\beta,\vartheta)$ such that there holds: For any $0<r<\rho$ and any non-negative bounded function $\phi\colon [r,\rho]\to\R_{\ge 0}$ satisfying
\begin{equation*}
	\phi(t)
	\le
	\vartheta \phi(s) + A\big (s^\alpha-t^\alpha\big)^{-\beta} + C
	\qquad \text{for all $r\le t<s\le \varrho$.}
\end{equation*}
we have
\begin{equation*}
	\phi(r)
	\le
	c \big[A (\varrho^\alpha - r^\alpha)^{-\beta} + C\big].
\end{equation*}
\end{lemma}

\subsection{Some elementary estimates}
We will frequently use estimates for the differences of powers,
cf. \cite[Lemma 8.3]{Giusti:book}. 

\begin{lemma}\label{lem:a-b}
Let $a,b\in\R^N$ and $\alpha>0$. Then we have that
\begin{align}\label{a-b-1}
 \tfrac1c\big|\power{a}{\alpha} - \power{b}{\alpha}\big|
	\le
	\big(|a| + |b|\big)^{\alpha-1}|a-b|
	\le
	c \big|\power{a}{\alpha} - \power{b}{\alpha}\big|,
	\mbox{ for some $c=c(\alpha)$.}
\end{align}
Moreover, if $0<\alpha<1$, then
\begin{align}\label{a-b-2}
        \big|\power{a}{\alpha} - \power{b}{\alpha}\big|
        \le
        c|a-b|^\alpha.
\end{align}
\end{lemma}

For the proof of the following lemma, we refer to \cite[Lemma 3.5]{BDKS:2018}.

\begin{lemma} \label{lem:mvint}
Let $p\ge1$ and $\alpha\ge \frac1 p$. Then there exists a universal constant $c=c(\alpha,p)$ such that if $A\subset B\subset \R^k$, $k\in\N$, are two bounded domains and $u \in L^{\alpha p}(B,\R)$, then there holds for any $s\in\R$
$$
	\mint_B \big|\power{u}{\alpha}-\power{(u)_A}{\alpha}\big|^p \dx 
	\leq 
	\frac{c\,|B|}{|A|} \mint_B \big|\power{u}{\alpha}-\power{s}{\alpha}\big|^p \dx.
$$

\end{lemma}

\subsection{Gagliardo-Nirenberg interpolation}

We note that for any $v\in K_\psi$, the properties $\power vm\in
L^\infty(0,T;L^{\frac{1+m}{m}}(\Omega))$ and $\power vm\in
  L^2(0,T;W^{1,2}(\Omega))$ imply by Gagliardo-Nirenberg interpolation
  that
  \begin{equation}
    \label{integrability-Kpsi}
    v\in L^{2m+\frac{2(1+m)}{n}}_{\mathrm{loc}}(\Omega_T),
  \end{equation}
  cf. \cite[Chapter 2, Prop. 4.1]{DBGV-book}. This represents an improvement compared
  to the integrability $v\in L^{1+m}(\Omega)$ 
  if and only if $m>\frac{(n-2)_+}{n+2}$,
  which underlines the significance of the critical exponent 
  $\frac{(n-2)_+}{n+2}$.

\section{An energy estimate and a Sobolev-Poincar\'e inequality}
\label{sec:energy-poincare}

The first step towards the proof of a reverse H\"older inequality is
the following energy estimate. 

\begin{lemma}\label{lem:energy}
Let $0<m\le1$ and $u\in K_\psi$ be a local weak solution to the
variational inequality \eqref{vari-ineq}, under assumptions
\eqref{obs-func}-\eqref{A-struc}.
Then there exists a constant $c=c(m,\nu,L)>0$ such that
on any cylinder $Q_{\rho}^{(\theta)}(z_o)\subset\Omega_T$ with
$0<\rho\le 1$ and $\theta>0$,  the following energy estimate holds:
\begin{align}\label{energy:a}\nonumber
	& \sup_{t \in \Lambda_r } 
 \mint_{B_r^{(\theta)} (x_o)} 
	\frac{ \big|\power u {\frac{1+m}2}(\cdot,t)-\power a{\frac{1+m}2}\big|^2}{r^{\frac{m+1}{m}}} \dx 
	+
	\biint_{Q_r^{(\theta)}(z_o)} \big|D\power um\big|^2 \dxdt \\\nonumber
	&\qquad\leq 
	c\,\biint_{Q_\rho^{(\theta)}(z_o)} 
	\Bigg[ \frac{\big|\power u {\frac{1+m}2} - \power a {\frac{1+m}2}\big|^2}{\rho^{\frac{1+m}{m}}-r^{\frac{1+m}{m}}}
	+
	\theta^{\frac{2m(1-m)}{1+m}}\frac{\big|\power um-\power a {m}\big|^2}{(\rho-r)^2}\Bigg]
	  \dxdt \\
	  &\qquad\qquad+
	c\, \biint_{Q_{\rho}^{(\theta)}(z_o)}\big( G+|a|^{1+m}\big) \dxdt ,
\end{align}
for every $r\in[\rho/2,\rho)$ and every  $a \in\R$,
where
\begin{align*}
G:= |F|^{2}  +|g|^{1+m}+ |D\power \psi m|^2 + |\partial_t \power \psi m|^{\frac{1+m}m} .
\end{align*}
In particular, we have the estimate
\begin{align}\label{energy:a=0}
	& \sup_{t \in \Lambda_{\rho/2}(t_o) } 
 \mint_{B_{\rho/2} (x_o)} 
	\frac{ |u(\cdot,t)|^{1+m}}{\rho^{\frac{1+m}{m}}} \dx 
	+
	\biint_{Q_{\rho/2}^{(\theta)}(z_o)} \big|D\power um\big|^2 \dxdt \nonumber\\
	&\qquad\leq 
	c\,\biint_{Q_\rho^{(\theta)}(z_o)} 
	\bigg[ \frac{|u|^{1+m}}{\rho^{\frac{1+m}{m}}}
	+
	\theta^{\frac{2m(1-m)}{1+m}}\frac{|u|^{2m}}{\rho^2} + G\bigg] \dxdt 
\end{align}
for some $c=c(m,\nu,L)>0$.
\end{lemma}

\begin{proof}
For notational convenience, we assume that $z_o$ is the origin, and we
omit the origin in our notation by writing $B_\rho^{(\theta)}:=B_\rho^{(\theta)}(0)$, $\Lambda_\rho:=\Lambda_\rho(0)$, and $Q_\rho^{(\theta)}:=Q_\rho^{(\theta)}(0)=B_\rho^{(\theta)}\times\Lambda_\rho$.
Let $r\in [\frac{\rho}2,\rho)$ and let
$\phi\in C^1_0\big(B^{(\theta)}_\rho,[0,1]\big)$ be the standard cut
off function with $\phi \equiv 1$ in $B_r^{(\theta)}$ and $|D\phi|\leq
\frac{2\theta^{\frac{m(1-m)}{1+m}}}{\rho-r}$. We also consider the functions
$\zeta \in W^{1,\infty} \big(\Lambda_\rho,[0,1]\big)$ and
$\xi_\varepsilon \in W^{1,\infty}\big (\Lambda_\rho,[0,1]\big)$
defined by 
\begin{equation*}
	\zeta(t)
	:=
	\left\{
	\begin{array}{cl}
	\frac{t + \rho^{\frac{1+m}{m}}}{\rho^{\frac{1+m}{m}}-r^{\frac{1+m}{m}}}, &
	\mbox{for $t\in \big(-\rho^{\frac{1+m}{m}},-r^{\frac{1+m}{m}}\big)$},\\[3pt]
	1 ,& \mbox{for $t\in \big[-r^{\frac{1+m}{m}},  \rho^{\frac{1+m}{m}}\big)$,}
	\end{array}
	\right.
\end{equation*}
and, for given $\varepsilon >0$ and $t_1 \in \Lambda_r$ 
$$
	\xi_\varepsilon(t) 
	:= 
	\left\{
	\begin{array}{cl}
	1 ,& \mbox{for $t\in \big(- \rho^{\frac{1+m}{m}}, t_1\big]$,} \\[3pt]
	1-\frac{1}{\varepsilon} (t-t_1) ,& 
	\mbox{for $t\in (t_1, t_1+\varepsilon)$,} \\[3pt]
	0 , & \mbox{for $t\in [t_1+\varepsilon,  \rho^{\frac{1+m}{m}} )$.} 
\end{array}
\right.
$$
Next, for a parameter $a\in\R$ we choose
$$
	\power{v}{m}
	:= 
	\max\{\power{a}{m}, \power{\psi}{m}\} = \power a m + (\power \psi m-\power a m)_+
$$
as comparison function in the variational inequality
\eqref{vari-ineq}. Note that this function is admissible in
\eqref{vari-ineq} since \eqref{obs-func} implies $v\in K_\psi'$. 
Moreover, we use the cut-off functions
$\alpha = \zeta\xi_{\varepsilon}$ and $\eta=\phi^2$ in
\eqref{vari-ineq}.
With these choices, the integral 
involving the time derivative in \eqref{vari-ineq} is given by
\begin{align}\label{10}\nonumber
&\langle\!\langle\partial_tu,(\zeta\xi_\varepsilon)\phi^2(\power vm- \power um)\rangle\!\rangle \\\nonumber
&=\iint_{Q_\rho^{(\theta)}} \phi^2 (\zeta\xi_\varepsilon)' \left( \tfrac1{1+m}|u|^{1+m} -u\power am \right) \dxdt \\
& \quad - \iint_{Q_\rho^{(\theta)}} \phi^2 (\zeta\xi_\varepsilon)' u
(\power \psi m-\power am)_+ \dxdt
-
\iint_{Q_\rho^{(\theta)}}\phi^2 \zeta\xi_\varepsilon u \partial_t(\power \psi m-\power am)_+ \dxdt.
\end{align}
At this point, it is convenient to introduce the boundary term
$\mathfrak{b}$ defined in \eqref{def:b}, which can be rewritten to
\begin{equation*}
  \mathfrak b[\power{u}{m},\power{a}{m}]
  :=
  \tfrac1{1+m}|u|^{1+m} -u\power am +\tfrac{m}{1+m} \abs{a}^{1+m}.
\end{equation*}
Since $\zeta\xi_\varepsilon \in W^{1,\infty}_0(\Lambda_\rho)$, we have
$$\iint_{Q_\rho^{(\theta)}} \phi^2 (\zeta\xi_\varepsilon)' \tfrac{m}{1+m}|a|^{1+m} \dxdt =0$$
and
$$\iint_{Q_\rho^{(\theta)}} \phi^2 (\zeta\xi_\varepsilon)' a (\power \psi m-\power am)_+\dxdt + \iint_{Q_\rho^{(\theta)}}\phi^2 \zeta\xi_\varepsilon a \partial_t(\power \psi m-\power am)_+ \dxdt =0.$$
Using these observations in \eqref{10}, we deduce 
\begin{align}\label{11}
\nonumber \langle\!\langle\partial_tu,(\zeta\xi_\varepsilon)\phi^2(\power wm-\power um)\rangle\!\rangle 
&=\iint_{Q_\rho^{(\theta)}} \phi^2 (\zeta\xi_\varepsilon)' \mathfrak b\big[\power{u}{m}, \power{a}{m}\big] \dxdt  \\\nonumber
&\quad - \iint_{Q_\rho^{(\theta)}} \phi^2 (\zeta\xi_\varepsilon)' (u-a) (\power \psi m-\power am)_+ \dxdt \\\nonumber
&\quad - \iint_{Q_\rho^{(\theta)}}\phi^2 \zeta\xi_\varepsilon (u-a) \partial_t(\power \psi m-\power am)_+ \dxdt\\
&=:\mathrm{I}_1 +\mathrm{I}_2+\mathrm{I}_3.
\end{align}
Letting $\eps\downarrow0$ and keeping in mind the definitions of
$\zeta$ and $\xi_\eps$, we infer that 
\begin{align}\label{12}
\limsup_{\epsilon\downarrow0}\mathrm{I}_1 
&\leq 
	\iint_{Q_\rho^{(\theta)}} 
	\frac{\mathfrak b\big[\power{u}{m}, \power{a}{m}\big]}{\rho^{\frac{1+m}{m}}-r^{\frac{1+m}{m}}} \dxdt
-\int_{B_\rho^{(\theta)}} \phi^2  \mathfrak b\big[\power{u}{m}(\cdot,t_1), \power{a}{m}\big] \dx
\end{align}
and
\begin{align}\label{15-1}
\limsup_{\epsilon\downarrow0}|\mathrm{I}_2 |
&\le
\iint_{Q_\rho^{(\theta)}} 
\frac{|u-a|(\power \psi m - \power am)_+}{\rho^{\frac{1+m}{m}}-r^{\frac{1+m}{m}}} \dxdt \\\nonumber
&\quad\quad+\int_{B_\rho^{(\theta)}} \phi^2  |u(\cdot, t_1)-a|(\power \psi m(\cdot, t_1) - \power am)_+ \dx.
\end{align}
In the last estimate, we use the fact $u\ge \psi$ a.e. in $\Omega_T$
and inequality \eqref{a-b-1} twice to deduce 
\begin{align}\label{est-I2-1}
|u-a|(\power \psi m - \power am)_+ &\le |u-a||\power um-\power am| \nonumber\\
&\le c\big((|u|+ |a|)^{\frac{m-1}2}|u-a|\big)^2 \nonumber\\
&\le c\big|\power u{\frac{1+m}2}- \power a{\frac{1+m}2}\big|^2 .
\end{align}
On the other hand, we can apply in turn \eqref{a-b-1}, the triangle
inequality,  and then \eqref{a-b-2} and \eqref{a-b-1}, with the result
\begin{align*}
	|u-a|(\power{\psi}{m}-\power{a}{m})_+ 
	& \le  
	c|\power um-\power am|\big(|\power um|+|\power am|\big)|^{\frac{1-m}{m}}(\power \psi m-\power am)_+ \\
	& \le  
	c|\power um-\power am|\big(|\power um-\power am|^{\frac{1-m}{m}}+|\power am|^{\frac{1-m}{m}}\big)(\power \psi m-\power am)_+ \\
	& \le  
	c|\power u{\frac{1+m}2}-\power a{\frac{1+m}2}|^{\frac2{1+m}}(\power \psi{\frac{1+m}2}-\power a{\frac{1+m}2})_+^{\frac{2m}{1+m}} \\
	&\quad\quad+c\big|\power{u}{\frac{1+m}{2}}-\power{a}{\frac{1+m}{2}}\big| \big(\power{\psi}{\frac{1+m}{2}}-\power{a}{\frac{1+m}{2}}\big)_+ .
\end{align*} 
Estimating the right-hand side further by Young's inequality and
\eqref{boundary-term}, we arrive at 
\begin{align}\label{est-I2-2}
|u-a|(\power{\psi}{m}-\power{a}{m})_+ 
&\le \tfrac12 \mathfrak{b}\big[ \power um, \power am\big]+c|\power\psi{\frac{1+m}2}-\power{a}{\frac{1+m}2}|^2.
\end{align}
We now estimate the two integrals in \eqref{15-1} by \eqref{est-I2-1}
and \eqref{est-I2-2}, respectively, which leads to the bound
\begin{align}\label{15-2}
	\limsup_{\epsilon\downarrow0}|\mathrm{I}_2 |
	&\le 
	c\iint_{Q_\rho^{(\theta)}} 
	\frac{\big| \power u {\frac{1+m}2} - \power a {\frac{1+m}2} \big|^2}{\rho^{\frac{1+m}{m}}-r^{\frac{1+m}{m}}} \dxdt 
	+
	\tfrac12\int_{B_\rho^{(\theta)}} \phi^2  \mathfrak b\big[\power{u}{m}(\cdot,t_1), \power{a}{m}\big] \dx\nonumber\\
	&\quad\quad+
	c\int_{B_\rho^{(\theta)}}  \big( \power \psi {\frac{1+m}2}(\cdot,t_1) - \power a {\frac{1+m}2} \big)_+^2 \dx.
\end{align}
For the last integral we again apply the obstacle constraint
$u\ge\psi$ a.e., inequality \eqref{a-b-1}, and then Young's inequality.
This yields the estimate
\begin{align}\label{aa}\nonumber
	&\int_{B_\rho^{(\theta)}}  \big( \power{\psi}{\frac{1+m}2}(\cdot,t_1) - \power a {\frac{1+m}2} \big)_+^2 \dx \\\nonumber
	& \quad \le 
	\mint_{\Lambda_\rho} \int_{B_\rho^{(\theta)} }  \big( \power \psi {\frac{1+m}2} - \power a {\frac{1+m}2} \big)_+^2 \dxdt
	+
	\int_{\Lambda_\rho}\int_{B_\rho^{(\theta)}} \Big| \partial_{t} \big( \power \psi {\frac{1+m}2} - \power a {\frac{1+m}2} \big)_+^2 \Big|  \dxdt  \\\nonumber
	 & \quad \le
	 \mint_{\Lambda_\rho} \int_{B_\rho^{(\theta)} }  \big| \power u {\frac{1+m}2} - \power a {\frac{1+m}2} \big|^2 \dxdt
	 +
	 c\iint_{Q_\rho^{(\theta)}} (  \psi - a )_+  |\partial_{t} \power \psi m| \dxdt \\
	 & \quad \le
	 \iint_{Q_\rho^{(\theta)} }  \frac{\big| \power u {\frac{1+m}2} - \power a {\frac{1+m}2} \big|^2}{\rho^{\frac{1+m}m} - r^{\frac{1+m}m}} \dxdt
	 +
	 c\iint_{Q_\rho^{(\theta)}}\big( | u-a|^{1+m} {+} |\partial_{t} \power \psi m|^{\frac{1+m}m} \big)\dxdt  . 
\end{align}
In the last step, we also used the fact $\rho\le1$.
For the estimate
of the last integral, we use inequality \eqref{a-b-1} and Young's inequality, which gives
\begin{align}\label{elementary}\nonumber
  |u-a|^{1+m}
  &\le
  c\Big[(|u-a|+|a|)^{\frac{1-m}2}\big|\power{u}{\frac{1+m}2}-\power{a}{\frac{1+m}2}\big|\Big]^{1+m}\\
  &\le\tfrac12|u-a|^{1+m}
  +
  c\big|\power{u}{\frac{1+m}2}-\power{a}{\frac{1+m}2}\big|^2+c|a|^{1+m}.
\end{align}
The first term on the right-hand side can be re-absorbed into the
left-hand side. Inserting the resulting estimate into \eqref{aa}, we infer
\begin{align*}
\nonumber
	&\int_{B_\rho^{(\theta)}}  \big( \power{\psi}{\frac{1+m}2}(\cdot,t_1) - \power a {\frac{1+m}2} \big)_+^2 \dx \\
	&\le 
	c\iint_{Q_\rho^{(\theta)}} 
	\frac{ \big|\power u {\frac{1+m}2} - \power a {\frac{1+m}2}\big|^2}{\rho^{\frac{1+m}{m}}-r^{\frac{1+m}{m}}} \dxdt 
        +
        c\iint_{Q_\rho^{(\theta)}}  \big( |\partial_{t} \power \psi m |^{\frac{1+m}m} +|a|^{1+m}\big)\dxdt.
\end{align*}
In view of \eqref{15-2} and the previous inequality, we finally obtain
\begin{align}\label{16}
\nonumber
	\limsup_{\epsilon\downarrow0}|\mathrm{I}_2|
	&\le 
	c\iint_{Q_\rho^{(\theta)}} 
	\frac{ \big|\power u {\frac{1+m}2} - \power a {\frac{1+m}2}\big|^2}{\rho^{\frac{1+m}{m}}-r^{\frac{1+m}{m}}} \dxdt 
	  +c\iint_{Q_\rho^{(\theta)}} \big( |\partial_{t} \power \psi m |^{\frac{1+m}m} +|a|^{1+m} \big) \dxdt\\
	&\quad\quad 
	+
	\tfrac12\int_{B_\rho^{(\theta)}} \phi^2 \mathfrak{b} \big[\power um (\cdot,t_1) , \power am\big] \dx.
\end{align}
Similarly, using Young's inequality and~\eqref{elementary}, we infer that
\begin{align}\label{17}
\nonumber
	|\mathrm{I}_3|
	&\le 
	\iint_{Q_\rho^{(\theta)}}  |u-a| |\partial_t \power \psi m|
        \dxdt \\\nonumber
        &\le
        \iint_{Q_\rho^{(\theta)}}\big( | u-a|^{1+m} {+} |\partial_{t} \power \psi m|^{\frac{1+m}m} \big)\dxdt\\
	&\le
	c\iint_{Q_\rho^{(\theta)}} 
	\frac{ \big|\power u {\frac{1+m}2} - \power a {\frac{1+m}2}\big|^2}{\rho^{\frac{1+m}{m}}-r^{\frac{1+m}{m}}} \dxdt 
	  +c\iint_{Q_\rho^{(\theta)}} \big( |\partial_{t} \power \psi m |^{\frac{1+m}m} +|a|^{1+m}\big)\dxdt .
\end{align}
Combining \eqref{12}, \eqref{16}, and \eqref{17} and applying
\eqref{boundary-term}, we finally arrive at
\begin{align}\label{est-time}
&  \limsup_{\eps\downarrow0}\,\langle\!\langle\partial_tu,(\zeta\xi_\varepsilon)\phi^2(\power
wm-\power um)\rangle\!\rangle \\\nonumber
&\qquad \le 
c\iint_{Q_\rho^{(\theta)}} 
\frac{ \big| \power u {\frac{1+m}2} - \power a {\frac{1+m}2}
  \big|^2}{\rho^{\frac{1+m}{m}}-r^{\frac{1+m}{m}}} \dxdt
+
c\iint_{Q_\rho^{(\theta)}}  \big(| \partial_{t} \power \psi m |^{\frac{1+m}m}+|a|^{1+m} \big)\dxdt\\\nonumber
&\qquad\qquad 
-
\frac12\int_{B_\rho^{(\theta)}} \phi^2 \mathfrak b \big[\power um (\cdot,t_1), \power am\big] \dx.
\end{align}
We next consider the diffusion term in the variational inequality
\eqref{vari-ineq}. For our choice $\power vm= \power
am+(\power\psi m-\power am)_+$, this term is given by 
\begin{align}\label{19}
& \iint_{Q_\rho^{(\theta)}}
\zeta\xi_\varepsilon \A \cdot D(\phi^2(\power vm-\power um)) \dxdt\\\nonumber
&\qquad =\iint_{Q_\rho^{(\theta)}}
\zeta\xi_\varepsilon \A \cdot D(\phi^2(\power am-\power um)) \dxdt  \\\nonumber
& \qquad\qquad +
\iint_{Q_\rho^{(\theta)}} \zeta\xi_\varepsilon \A \cdot
D(\phi^2(\power \psi m-\power am)_+) \dxdt\\\nonumber
&\qquad=:\mathrm{II}_1 + \mathrm{II}_2.
\end{align}
Adopting the growth condition in \eqref{A-struc} and Young's
inequality, we obtain 
\begin{align*}
	\mathrm{II}_1
	&= \iint_{Q_\rho^{(\theta)}} \zeta\xi_\epsilon \A \cdot \big(2
        \phi  (\power am-\power um)D\phi  - \phi^2 D\power um \big) \dxdt \\
	&\leq  
	2L\iint_{Q_\rho^{(\theta)}} 
	\zeta \xi_\varepsilon \phi  |D\phi| \big|\power u{m}-\power
        a{m}\big|\big|D\power u{m}\big|
	\dx\dt 
	-\nu\iint_{Q_\rho^{(\theta)}} 
	\zeta \xi_\varepsilon \phi^2  \big|D\power u{m}\big|^2
          \dx\dt
	\\
	&\leq   
	c \iint_{Q_\rho^{(\theta)}} \theta^{\frac{2m(1-m)}{1+m}}
        \frac{\big|\power u{m}-\power a{m}\big|^2}{(\rho-r)^2} \dx\dt -
	\frac{\nu}{2} \iint_{Q_\rho^{(\theta)}} 
	 \zeta \xi_\varepsilon  \phi^2 \big|D\power u{m}\big|^2 \dx\dt,
\end{align*}
and
\begin{align*}
\mathrm{II}_2 
&= \iint_{Q_\rho^{(\theta)}} \zeta\xi_\epsilon \A \cdot \big(2 \phi
(\power\psi m-\power am)_+  D\phi + \phi^2 D(\power \psi m-\power am)_+ \big) \dxdt \\
&\le	2L\iint_{Q_\rho^{(\theta)}} 
	\zeta \xi_\varepsilon \phi  |D\phi| \big(\power\psi{m}-\power
        a{m}\big)_+ \big|D\power u{m}\big|
	\dxdt \\
&\qquad	+ L \iint_{Q_\rho^{(\theta)}} 
	\zeta \xi_\varepsilon \phi^2  \big|D\power u{m}\big| \big|
        D(\power \psi m-\power am)_+\big|
          \dxdt \\
& \le c\iint_{Q_\rho^{(\theta)}} 
	   \theta^{\frac{2m(1-m)}{1+m}}\frac{ \big|\power u{m}-\power a{m}\big|^2}{(\rho-r)^2} \dxdt
	+ c \iint_{Q_\rho^{(\theta)}} \big|D\power \psi{m}\big|^2 \dx\dt \\
&\quad\quad	+\frac{\nu}{4} \iint_{Q_\rho^{(\theta)}} 
	 \zeta \xi_\varepsilon  \phi^2 \big|D\power u{m}\big|^2 \dx\dt.
\end{align*}
In the last step, we also used the obstacle constraint $u\ge\psi$
a.e. in $\Omega_T$.        
Plugging the two preceding estimates into~\eqref{19}, we conclude
\begin{align}\label{est-A}
&\iint_{Q_\rho^{(\theta)}} \zeta\xi_\varepsilon \A \cdot
D(\phi^2(\power vm-\power um)) \dxdt \\\nonumber
& \qquad \le c\iint_{Q_\rho^{(\theta)}} 
	   \theta^{\frac{2m(1-m)}{1+m}}\frac{ \big|\power u{m}-\power a{m}\big|^2}{(\rho-r)^2} \dxdt
	+ c \iint_{Q_\rho^{(\theta)}} \big|D\power\psi{m}\big|^2 \dx\dt \\\nonumber
&\qquad\qquad	-\frac{\nu}{4} \iint_{Q_\rho^{(\theta)}} 
	 \zeta \xi_\varepsilon  \phi^2 \big|D\power u{m}\big|^2 \dx\dt
\end{align}
for some $c=c(m,\nu,L)>0.$
We now estimate the right-hand side terms in the variational
inequality \eqref{vari-ineq}. For the divergence term, we use the
bounds $|\power vm-\power um|\le|\power um-\power am|$ and $|D(\power
vm-\power um)\le |D\power um|+|D\power\psi m|$ and Young's inequality,
with the result 
\begin{align}\label{est-F}
\nonumber
&-\iint_{Q_{\rho}^{(\theta)}} \zeta\xi_\varepsilon  F \cdot
D\big(\phi^2 \big(\power vm-\power um\big)\big) \dxdt \\\nonumber
&\qquad\le 
\iint_{Q_{\rho}^{(\theta)}} \zeta\xi_\varepsilon  |F|
\big(2\phi |D\phi||\power vm-\power um|+\phi^2|D(\power vm-\power um)|\big) \dxdt\\\nonumber
&\qquad\ge
	c \iint_{Q_{\rho}^{(\theta)}}
        \theta^{\frac{2m(1-m)}{1+m}}\frac{\big|\power u{m}-\power a{m}\big|^{2}}{(\rho-r)^2} \dxdt 
	+ c \iint_{Q_{\rho}^{(\theta)}} |F|^{2}\dxdt \\
	&\qquad\qquad
	+ c \iint_{Q_{\rho}^{(\theta)}} |D\power\psi m|^{2}\dxdt
        +\frac{\nu}{8}
	\iint_{Q_{\rho}^{(\theta)}} 
	 \zeta \xi_\varepsilon\phi^2 \big|D\power u{m}\big|^2 \dxdt.
\end{align}
Finally, we use the fact $|\power vm-\power um|\le|\power um-\power
am|$ and then Young's inequality, \eqref{a-b-2} and the assumption
$\rho\le1$ for the estimate
\begin{align}\label{est-g}
\nonumber
 &-\iint_{Q_{\rho}^{(\theta)}} \zeta\xi_\varepsilon \phi^2 g (\power
vm-\power um) \dxdt 
\\\nonumber
&\qquad\le \iint_{Q_{\rho}^{(\theta)}} |g| \big|\power um-\power am\big| \dxdt \\
& \qquad \le c \iint_{Q_{\rho}^{(\theta)}} \frac{\big|\power u{\frac{1+m}2}-\power a{\frac{1+m}2}\big|^2}{\rho^{\frac{1+m}m}-r^{\frac{1+m}m}} \dxdt
+c \iint_{Q_{\rho}^{(\theta)}}  |g|^{1+m} \dxdt.
\end{align} 
Now, we use \eqref{est-time}, \eqref{est-A}, \eqref{est-F}, and
\eqref{est-g} to pass to the limit $\epsilon\downarrow 0$ in the
variational inequality \eqref{vari-ineq}. In this way, we deduce 
\begin{align*}
	&\frac12\int_{B_r^{(\theta)}} \mathfrak b\big[ \power um(\cdot, t_1), \power am\big] \dx +  
	\frac{\nu}{8}
	\int_{- r^{\frac{1+m}{m}}}^{t_1}\int_{B_r^{(\theta)}}  
	\big|D\power u{m}\big|^2 \dxdt \\
	&\quad\leq 
	c \iint_{Q_\rho^{(\theta)}}
	\bigg[
	\frac{\big| \power u {\frac{1+m}2}-\power a {\frac{1+m}2}\big|^2}{\rho^{\frac{1+m}{m}}-r^{\frac{1+m}{m}}}
	+ 
	\theta^{\frac{2m(1-m)}{1+m}} \frac{\big|\power u{m}-\power a{m}\big|^2}{(\rho-r)^2} 
	\bigg]  
	\dxdt \\
	&\quad\quad + c \iint_{Q_{\rho}^{(\theta)}}\big( |F|^{2}
        +\big|D\power \psi m \big|^2 + |\partial_t \power \psi m|^{\frac{1+m}{m}} + |g|^{1+m}+ |a|^{1+m}\big) \dxdt,
\end{align*}
for any $t_1 \in \Lambda_r$, where the constant $c>0$ depends on $m,
\nu,$ and $L$.
Finally, we take the supremum over $t_1 \in \Lambda_r$ on the
left-hand side and take means. In view of \eqref{boundary-term} this
leads to the result \eqref{energy:a}. The second assertion
\eqref{energy:a=0} then follows by letting $a=0$ and $r=\frac\rho2$.
\end{proof}

Our next goal is the proof of a so-called gluing lemma, which enables
us to compare mean values of the solution taken on different time slices.

\begin{lemma}\label{lem:time-diff}
  Let $Q_{\rho}^{(\theta)}(z_o)\subset\Omega_T$ with 
$0<\rho\le 1$, $\theta>0$, and $z_o\equiv(x_o,t_o)$. 
Under the assumptions of Lemma~\ref{lem:energy}, there exist $\hat{\rho}\in\big[\frac\rho2,\rho\big]$ and $c=c(n,m,\nu,L)>0$ such that
\begin{align}\label{gluing}\nonumber
&|\<u\>_{x_o;\hat{\rho}}(t_1)-\<u\>_{x_o;\hat{\rho}}(t_2)|\\\nonumber
&\qquad\le
c\rho^{\frac1m}\theta^{\frac{m(1-m)}{m+1}}\biint_{Q_\rho^{(\theta)}(z_o)}
\big(|D\power um|+|F|\big)\dxdt+c\mu \\\nonumber
&\qquad\quad+
c\mu^m\bigg[\biint_{Q_\rho^{(\theta)}(z_o)} \big(|u|^{1+m} 
+
\rho^{\frac{1-m}m}\theta^{\frac{2m(1-m)}{1+m}} |u|^{2m}\big)
\dxdt\bigg]^{\frac{1-m}{1+m}}\\\nonumber
&\qquad\quad+
c\rho^{\frac{1+m}m}\bigg[\biint_{Q_\rho^{(\theta)}(z_o)} |u|^{1+m} \dxdt \bigg]^{\frac{1-m}{2(1+m)}} \\\nonumber
&\qquad\qquad\qquad\times
\bigg[\biint_{Q_\rho^{(\theta)}(z_o)}\big(|u|^{1+m}+ |\partial_t \power \psi m|^{\frac{1+m}m}\big) \dxdt\bigg]^{\frac12} \\
&\qquad\quad+
  \frac{c\rho^{\frac{1+m}m}}{\mu^m}
  \biint_{Q_\rho^{(\theta)}(z_o)} G \dxdt
+
c\mu^m\bigg[\rho^{\frac{1+m}m}\biint_{Q_{\rho}^{(\theta)}(z_o)}
G \dxdt\bigg]^{\frac{1-m}{1+m}}
\end{align}
for all $t_1,t_2\in\Lambda_\rho(t_o)$,
where $\mu>0$ is an arbitrary parameter
and we used the abbreviation
\begin{equation}\label{def-G-2}
  G:=|F|^2+|D\power\psi m|^2+|\partial_t\power\psi m|^{\frac{1+m}m}+|g|^{1+m}.
\end{equation}
\end{lemma}

\begin{proof}
For notational convenience, we assume that $z_o$ is the origin, and we
omit the origin in our notation by writing $B_\rho^{(\theta)}:=B_\rho^{(\theta)}(0)$, $\Lambda_\rho:=\Lambda_\rho(0)$, and $Q_\rho^{(\theta)}:=Q_\rho^{(\theta)}(0)=B_\rho^{(\theta)}\times\Lambda_\rho$. 
Moreover, without loss of generality we suppose that
$t_1,t_2\in\Lambda_\rho$ are ordered in the way that $t_1<t_2$ holds true.
For $r \in [\frac\rho2, \rho]$, $0<\delta\ll1$, and $0<\varepsilon\ll 1$, we
define a cut-off function 
$\alpha\in W^{1,\infty}_0(\Lambda_\rho)$ in time by letting
$$
	\alpha(t)
	:=
	\left\{
	\begin{array}{cl}
	0 ,& \mbox{for $-\rho^{\frac{1+m}{m}}< t\le t_1-\epsilon$,}\\[3pt]
	\frac{t-t_1+\epsilon}{\epsilon}, & \mbox{for $t_1-\epsilon< t< t_1$,}\\[3pt]
	1 ,& \mbox{for $t_1\le t\le t_2$,}\\[3pt]
	\frac{t_2+\epsilon-t}{\epsilon}, & \mbox{for $t_2< t< t_2+\epsilon$,}\\[3pt]
	0 ,& \mbox{for $t_2+\epsilon\le t < \rho^{\frac{1+m}{m}}$,}
	\end{array}
	\right.
$$
and a radial cut-off function in space $\eta \in W^{1,\infty}_0(B_\rho^{(\theta)})$ by 
$$
	\eta(x)
	:=
	\left\{
	\begin{array}{cl}
	1 ,& \mbox{for $0\le |x|\le r\theta^{\frac{m(m-1)}{1+m}}$,}\\[3pt]
	\frac1\delta\big(r\theta^{\frac{m(m-1)}{1+m}}+\delta-|x|\big), & \mbox{for $r\theta^{\frac{m(m-1)}{1+m}}< |x|< r\theta^{\frac{m(m-1)}{1+m}}+\delta$,}\\[3pt]
	0 ,& \mbox{for $r\theta^{\frac{m(m-1)}{1+m}}+\delta\le |x| < \rho\theta^{\frac{m(m-1)}{1+m}}$.}
	\end{array}
	\right.
$$
We next define a comparison function $\power {v_h} m$ for the variational inequality \eqref{vari-ineq} by
\begin{equation*}
  \power{v_h}{m}:=\max\{\llbracket
  \power{u}{m}\rrbracket_h-\mu^m,\power{\psi}{m}\}
  =
  \llbracket \power{u}{m}\rrbracket_h-\mu^m+\big(\power{\psi}{m}+\mu^m-\llbracket
  \power{u}{m}\rrbracket_h\big)_+\,,
\end{equation*}
for a parameter $\mu>0$, where we used the time mollification defined
in~\eqref{time-molli}. With this choice of the comparison function,
the variational inequality implies 
\begin{align}\label{var-in-obst}
  0&\le
  \iint_{\Omega_T}\big[\eta\alpha'\big(\tfrac1{1+m}|u|^{1+m}-u\power{v_h}{m}\big)
  -\eta\alpha u\partial_t\power{v_h}{m}\big] \dxdt\\\nonumber
  &\qquad+
  \iint_{\Omega_T}\alpha \mathbf{A}(x,t,u,D\power{u}{m})\cdot
  D(\eta(\power{v_h}{m}-\power um))\dxdt\\\nonumber
  &\qquad - \iint_{\Omega_T} \alpha F(x,t) \cdot D(\eta(\power{v_h}{m}-\power um))\dxdt\\\nonumber
 &\qquad - \iint_{\Omega_T}  \alpha\eta g(x,t) (\power{v_h}{m}-\power um) \dxdt \\\nonumber
  &=:\mathrm{I}_h+\mathrm{II}_h+\mathrm{III}_h+\mathrm{IV}_h.
\end{align}
First, we note that the term $\mathrm{I}_h$ can be replaced by the
modified version 
\begin{align}\label{I-h-ast-01}
  \mathrm{I}^\ast_h
  :=
  \iint_{\Omega_T}\big[\eta\alpha'\big(\tfrac1{1+m}|u|^{1+m}-u\power{v_h}{m}\big)-\eta\alpha \power{\llbracket u^m\rrbracket_h}{\frac1m}\partial_t\power{v_h}{m}\big]\dxdt.
\end{align}
In fact, using the identity $\partial_t\llbracket
\power um\rrbracket_h=\frac1h(u-\llbracket \power um\rrbracket_h)$, we can
rewrite the difference of the two terms as follows. 
\begin{align*}
  \mathrm{I}_h-\mathrm{I}_h^\ast
  &=
  \iint_{\Omega_T}\eta\alpha(\power{\llbracket u^m\rrbracket_h}{\frac1m}-u)\partial_t\power{v_h}{m}\dxdt\\
  &=
  \iint_{\Omega_T\cap\{\llbracket\power{u}{m}\rrbracket_h>\power\psi
    m+\mu^m\}}
  \eta\alpha(\power{\llbracket u^m\rrbracket_h}{\frac1m}-u)\tfrac1h(\power um-\llbracket\power{u}{m}\rrbracket_h)\dxdt\\
  &\qquad+
  \iint_{\Omega_T\cap\{\llbracket\power{u}{m}\rrbracket_h\le\power\psi
    m+\mu^m\}}\eta\alpha(\power{\llbracket u^m\rrbracket_h}{\frac1m}-u)\partial_t\power{\psi}{m}\dxdt\\
  &\le
  \iint_{\Omega_T}\eta\alpha\big|\power{\llbracket u^m\rrbracket_h}{\frac1m}-u\big||\partial_t\power{\psi}{m}|\dxdt\\
  &\to0
\end{align*}
in the limit $h\downarrow0$, because $\partial_t\power\psi m\in
L^{\frac{1+m}m}(\Omega_T)$ and $\llbracket\power{u}{m}\rrbracket_h\to\power um$ in $L^{\frac{1+m}m}(\Omega_T)$. 
For the estimate of the 
 term $\mathrm{I}_h^\ast$, we apply the definition of $\power{v_h}{m}$, the chain rule, and
an integration by parts, with the result 
\begin{align}\label{I-h-ast}
  \mathrm{I}_h^\ast
  &=
   \iint_{\Omega_T}\eta\alpha'\Big[\tfrac1{1+m} |u|^{1+m}-u\llbracket\power{u}{m}\rrbracket_h+u\mu^m-u\big(\power\psi
  m+\mu^m-\llbracket\power{u}{m}\rrbracket_h\big)_+\Big]\dxdt\nonumber\\\nonumber
  &\qquad-
  \iint_{\Omega_T}\eta\alpha\tfrac{m}{1+m}
  \partial_t \big|\llbracket\power{u}{m}\rrbracket_h\big|^{\frac{1+m}{m}}\dxdt\\\nonumber
 &\qquad-
   \iint_{\Omega_T}\eta\alpha\power{\llbracket u^m\rrbracket_h}{\frac1m}\partial_t\big(\power\psi
  m+\mu^m-\llbracket\power{u}{m}\rrbracket_h\big)_+\dxdt\\\nonumber
  &=
  \mu^m\iint_{\Omega_T}\eta\alpha'u \dxdt\\\nonumber
  &\qquad
  +\iint_{\Omega_T}\eta\alpha'\big(\tfrac1{1+m}
  |u|^{1+m}-u\llbracket\power{u}{m}\rrbracket_h+\tfrac{m}{1+m}\big| \llbracket\power{u}{m}\rrbracket_h \big|^{\frac{1+m}{m}}\big)\dxdt\\\nonumber
  &\qquad
  -\iint_{\Omega_T}\eta\alpha' u\big(\power\psi m+\mu^m-
  \llbracket\power{u}{m}\rrbracket_h\big)_+\dxdt\\\nonumber
  &\qquad+
  \iint_{\Omega_T\cap\{\llbracket\power{u}{m}\rrbracket_h\le\power\psi
    m+\mu^m\}}\eta\alpha\power{\llbracket u^m\rrbracket_h}{\frac1m}\big(\partial_t\llbracket\power{u}{m}\rrbracket_h-\partial_t\power\psi
  m\big)\dxdt\\
  &=:\mathrm{I_{1}}+\mathrm{I_{2,h}}+\mathrm{I_{3,h}}+\mathrm{I_{4,h}}.
\end{align}
Because $\llbracket\power{u}{m}\rrbracket_h$ converges to
$\power um$ in $L^{\frac{1+m}m}(\Omega_T)$ as $h\downarrow0$, we observe
\begin{equation}\label{convergence-I-2}
  \lim_{h\downarrow0}\mathrm{I}_{2,h}=0
\end{equation}
and 
\begin{equation}\label{convergence-I-3}
  \lim_{h\downarrow0}\mathrm{I}_{3,h}
  =
  -\iint_{\Omega_T}\eta\alpha'u\big(\power\psi m+\mu^m-
  \power{u}{m}\big)_+\dxdt.
\end{equation}
Next, we rewrite the term $\mathrm{I}_{4,h}$ as
\begin{align*}
  \mathrm{I}_{4,h}
  &=
  \iint_{\Omega_T\cap\{\llbracket\power{u}{m}\rrbracket_h\le\power\psi
    m+\mu^m\}}
  \eta\alpha\Big[\power{\llbracket u^m\rrbracket_h}{\frac1m}
  \partial_t\llbracket\power{u}{m}\rrbracket_h
  -\power{(\psi^m+\mu^m)}{\frac1m}\partial_t\power\psi
  m\Big]\dxdt \\
  &\qquad +
  \iint_{\Omega_T\cap\{\llbracket\power{u}{m}\rrbracket_h\le\power\psi
    m+\mu^m\}}
  \eta\alpha\partial_t\power\psi m\Big[\power{(\psi^m+\mu^m)}{\frac1m}
  -\power{\llbracket u^m\rrbracket_h}{\frac1m}
  \Big]\dxdt\\
  &=
  -\iint_{\Omega_T\cap\{\llbracket\power{u}{m}\rrbracket_h\le\power\psi
    m+\mu^m\}}\eta\alpha\tfrac m{1+m} \partial_t\Big(\big|\power\psi
  m+\mu^m\big|^{\frac{1+m}m}
  -\big| \llbracket \power um\rrbracket_h \big|^{\frac{1+m}m}\Big)\dxdt\\
  &\qquad +
  \iint_{\Omega_T}
  \eta\alpha\partial_t\power\psi m\Big[\power{(\psi^m+\mu^m)}{\frac1m}
  -\power{\llbracket u^m\rrbracket_h}{\frac1m}
  \Big]_+\dxdt\\
    &=
  -\iint_{\Omega_T}\eta\alpha\tfrac m{1+m} \partial_t\Big[ \big| \llbracket \power um \rrbracket_h + (\power\psi
  m+\mu^m-\llbracket \power um \rrbracket_h)_+\big|^{\frac{1+m}m}
  -\big| \llbracket \power um\rrbracket_h \big|^{\frac{1+m}m}\Big]\dxdt\\
  &\qquad +
  \iint_{\Omega_T}
  \eta\alpha\partial_t\power\psi m\Big[\power{(\psi^m+\mu^m)}{\frac1m}
  -\power{\llbracket u^m\rrbracket_h}{\frac1m}
  \Big]_+\dxdt,
\end{align*}
where we used the chain rule for the second step.
We now perform an integration by parts and then let $h\downarrow0$,
which yields
\begin{align}\label{convergence-I-4}
\nonumber
  \lim_{h\downarrow0}\mathrm{I}_{4,h}
 &=
  \iint_{\Omega_T\cap\{\power um \le \power \psi m + \mu^m\}}\eta\alpha' \tfrac m{1+m} \big(|\power\psi m+\mu^m|^{\frac{1+m}m}-|u|^{1+m}
  \big)\dxdt\\
  &\qquad+\iint_{\Omega_T}
  \eta\alpha \partial_t\power\psi m \big(\power{(\psi^m+\mu^m)}{\frac1m}
  -u \big)_+\dxdt.
\end{align}
For the last integral, we deduce from $u\ge\psi$ a.e. that
\begin{align}\label{on_the_set_psi}
|\psi|^m \le |u|^m + |\power \psi m-\power um| \le |u|^m+\mu^m
\quad\mbox{a.e.}
\end{align}
on the set $\{\power um\le\power\psi m+\mu^m\}$.
By applying this and \eqref{a-b-1} we get
\begin{align*}
\nonumber
	\big( \power{(\psi^m +\mu^m)}{\frac1m} - \power{(u^m)}{\frac1m} \big)_+
	& \le
	c\big( | \psi |^m +|\mu|^m + | u|^m \big)^{\frac1m-1}(\power \psi m + \mu^m -\power um)_+ \\\nonumber
	& \le
	c\big( |u|^{ m}+\mu^m\big)^{\frac1m-1}\mu^m\\
	&\le
	c\mu^m |u|^{1-m} + c\mu,
\end{align*}
which implies by H\"older's inequality and Young's inequality that
\begin{align*}
	&\iint_{\Omega_T} \eta\alpha \partial_t \power \psi m  \big(\power{(\psi^m+\mu^m)}{\frac1m}
  -u \big)_+\dxdt \\\nonumber
	&\quad\quad\le
	c\mu^m\iint_{\Omega_T\cap\{\power um \le \power \psi m + \mu^m\}} \eta\alpha |u|^{\frac{1-m}2}|u|^{\frac{1-m}2}\big| \partial_t \power \psi m\big| \dxdt \\\nonumber
	&\quad\quad\quad\quad+
	c\mu\iint_{\Omega_T\cap\{\power um \le \power \psi m + \mu^m\}} \eta\alpha \big|\partial_t \power \psi m\big| \dxdt\\\nonumber
	&\quad\quad\le
	c\mu^m \bigg[\iint_{\Omega_T\cap\{\power um \le \power \psi m + \mu^m\}}\eta\alpha |u|^{1-m} \dxdt \bigg]^{\frac12}\\\nonumber
	&\quad\quad\quad\quad\times
	\bigg[\iint_{\Omega_T\cap\{\power um \le \power \psi m + \mu^m\}}\eta\alpha |u|^{1-m} |\partial_t \power \psi {m}|^{2} \dxdt\bigg]^{\frac{1}{2}} \\\nonumber
	&\quad\quad\quad+
	c\mu\iint_{\Omega_T\cap\{\power um \le \power \psi m + \mu^m\}} \eta\alpha \big|\partial_t \power \psi m\big| \dxdt\\\nonumber
	&\quad\quad\le
	c\mu^m|Q_\rho^{(\theta)}| \bigg[\biint_{Q_\rho^{(\theta)}} |u|^{1+m} \dxdt \bigg]^{\frac{1-m}{2(1+m)}}
	\bigg[\biint_{Q_\rho^{(\theta)}} \big( |u|^{1+m} + |\partial_t \power \psi {m}|^{\frac{1+m}m}\big) \dxdt\bigg]^{\frac{1}{2}} \\\nonumber
	&\quad\quad\quad\quad+
	c\mu\iint_{Q_\rho^{(\theta)}}\alpha\eta \big|\partial_t \power \psi m\big| \dxdt.
\end{align*}
We then insert this into \eqref{convergence-I-4} to obtain
\begin{align}\label{convergence-I-4-1}
 \lim_{h\downarrow0}\mathrm{I}_{4,h}
  &\le
  \iint_{\Omega_T\cap\{\power um \le \power \psi m + \mu^m\}}\eta\alpha' \tfrac m{1+m} \big(|\power\psi m+\mu^m|^{\frac{1+m}m}-|u|^{1+m}
  \big)\dxdt\\\nonumber
  &\qquad+
  c\mu^m|Q_\rho^{(\theta)}| \bigg[\biint_{Q_\rho^{(\theta)}} |u|^{1+m} \dxdt \bigg]^{\frac{1-m}{2(1+m)}}\\\nonumber
	&\quad\quad\quad\quad\times
	\bigg[\biint_{Q_\rho^{(\theta)}} \big( |u|^{1+m} + |\partial_t \power \psi {m}|^{\frac{1+m}m}\big) \dxdt\bigg]^{\frac{1}{2}} \\\nonumber
	&\quad\quad+
	c\mu\iint_{Q_\rho^{(\theta)}}\alpha\eta\big|\partial_t \power \psi m\big| \dxdt.
\end{align}
Collecting \eqref{I-h-ast-01}-\eqref{convergence-I-3}, and
\eqref{convergence-I-4-1}, and keeping in mind the definition of
$\mathfrak b$ in~\eqref{boundary-term}, we conclude
\begin{align}\label{limsup-I-h}
  &\limsup_{h\downarrow0}\mathrm{I}_h
  \le
  \lim_{h\downarrow0}\mathrm{I}_h^\ast\\\nonumber
  &\qquad\le
  \mu^m\iint_{\Omega_T}\eta\alpha'u\,\dxdt
  +\iint_{\Omega_T\cap\{\power um \le \power \psi m +\mu^m\}}\eta\alpha'\mathfrak{b}[\power um,\power\psi
  m+\mu^m]\dxdt\\\nonumber
  &\qquad\qquad+
  c\mu^m|Q_\rho^{(\theta)}| \bigg[\biint_{Q_\rho^{(\theta)}} |u|^{1+m} \dxdt \bigg]^{\frac{1-m}{2(1+m)}}\\\nonumber
	&\qquad\qquad\qquad\times
	\bigg[\biint_{Q_\rho^{(\theta)}} \big( |u|^{1+m} + |\partial_t \power \psi {m}|^{\frac{1+m}m}\big) \dxdt\bigg]^{\frac{1}{2}} \\\nonumber
	&\qquad\qquad+
	c\mu\iint_{Q_\rho^{(\theta)}}  \big|\partial_t \power \psi m\big| \dxdt.
\end{align}
We next estimate the term $\mathrm{II}_h$. Thanks to the
convergence properties of the time mollification, we infer 
\begin{align*}
  \lim_{h\downarrow0}\mathrm{II}_h 
  &\ =
   \iint_{\Omega_T}\alpha
    \mathbf{A}(x,t,u,D\power um)\cdot D\eta (-\mu^m+(\power\psi
    m+\mu^m-\power um)_+)\dxdt\\
    &\qquad+
    \iint_{\Omega_T}\alpha
    \eta \mathbf{A}(x,t,u,D\power um)\cdot D(\power\psi m+\mu^m-\power um)_+\dxdt
   \\
  &\le
  \mu^m\iint_{\Omega_T}\alpha |\mathbf{A}(x,t,u,D\power um)|\,|D\eta|\,\dxdt\\
  &\qquad
  +
  \iint_{\Omega_T\cap\{\power um \le \power \psi m +\mu^m\}}
  \alpha\eta \mathbf{A}(x,t,u,D\power um)\cdot D(\power\psi m-\power um)\dxdt,
\end{align*}
where we used the obstacle constraint $u\ge\psi$ a.e. for the last estimate.
Next, we employ the growth and coercivity assumption \eqref{A-struc},
and then Young's inequality in order to obtain 
\begin{align}\label{convergence-II-2}
  \lim_{h\downarrow0}\mathrm{II}_h
  &\le
  \mu^mL\iint_{\Omega_T}\alpha |D\power um|\,|D\eta|\,\dxdt\\\nonumber
  &\qquad+
  \frac{L^2}{2\nu}\iint_{\Omega_T\cap\{\power um \le \power \psi m +\mu^m\}}\alpha\eta  |D\power\psi m|^2\dxdt \\\nonumber
  &\qquad-
  \frac{\nu}{2}\iint_{\Omega_T\cap\{\power um \le \power \psi m +\mu^m\}}\alpha\eta  |D\power um|^2\dxdt.
\end{align}
The term  $\mathrm{III}_h$ is estimated 
in a similar way as $\mathrm{II}_h$, which yields 
\begin{align}\label{III-h}
  \lim_{h\downarrow0}\mathrm{III}_h 
 &\ =
   -\iint_{\Omega_T}\alpha
   F(x,t)\cdot D\eta (-\mu^m+(\power\psi
    m+\mu^m-\power um)_+)\dxdt \nonumber\\
    &\qquad-
    \iint_{\Omega_T}\alpha
    \eta F(x,t)\cdot D(\power \psi m+\mu^m-\power um)_+\dxdt \nonumber\\
   &\ \le
   \mu^m \iint_{\Omega_T} \alpha |F||D\eta| \dxdt + 
   \frac{\nu+1}{2\nu}\iint_{\Omega_T\cap\{\power um\le\power \psi m+\mu^m\}} \alpha \eta |F|^2\dxdt \nonumber\\\nonumber
   &\qquad 
   +\frac12 \iint_{\Omega_T\cap\{\power um\le\power \psi m+\mu^m\}} \alpha \eta |D\power\psi m|^2 \dxdt \\
   &\qquad+ 
   \frac{\nu}2\iint_{\Omega_T\cap\{\power um\le\power\psi m+\mu^m\}} \alpha \eta |D\power um|^2 \dxdt.
\end{align}
Next, the last term $\mathrm{IV}_h$ is estimated by
\begin{align}\label{IV-h}
  \lim_{h\downarrow0}\mathrm{IV}_h
  &\ =
  -\iint_{\Omega_T} \alpha \eta g (-\mu^m + (\power\psi m + \mu^m - \power um)_+) \dxdt \\\nonumber
  &\ \le
  \mu^m \iint_{\Omega_T} \alpha \eta |g| \dxdt.
\end{align}
According to \eqref{limsup-I-h}-\eqref{IV-h},
the variational inequality \eqref{var-in-obst} implies that
\begin{align*}
	0&\le
	\mu^m\iint_{\Omega_T}\eta\alpha'u\,\dx\dt
	+ c\mu^m\iint_{\Omega_T}\alpha (|D\power um|+|F|)|D\eta|\,\dxdt\\
	&\qquad
	+\iint_{\Omega_T\cap\{\power um\le\power\psi m+\mu^m\}}
	\eta\alpha'\mathfrak{b}[\power um,\power\psi m+\mu^m]\dxdt\\
	&\qquad
	+
	c\mu^m|Q_\rho^{(\theta)}| \bigg[\biint_{Q_\rho^{(\theta)}} |u|^{1+m} \dxdt \bigg]^{\frac{1-m}{2(1+m)}}\\
	&\quad\quad\quad\times
	\bigg[\biint_{Q_\rho^{(\theta)}} \big( |u|^{1+m} + |\partial_t \power \psi {m}|^{\frac{1+m}m}\big) \dxdt\bigg]^{\frac{1}{2}}\\
	&\qquad
	+
	c\iint_{\Omega_T}\alpha\eta\big(\mu\big|\partial_t\power\psi m\big|+|D\power\psi m|^2 + |F|^2+\mu^m|g|\big) \dxdt
\end{align*}
for some $c=c(m,\nu,L)>0$.
By means of Young's inequality, the last integrand can be estimated
by
  \begin{equation*}
    \mu\big|\partial_t\power\psi m\big|+|D\power\psi m|^2 +
    |F|^2+\mu^m|g|
    \le
    G+\mu^{1+m},
  \end{equation*}
with the function $G$ defined in \eqref{def-G-2}.
Recalling the definitions of the cut-off
functions $\alpha$ and $\eta$ and passing to the limit
$\delta\downarrow0$ and $\eps\downarrow0$, we deduce
\begin{align}\label{pre-gluing-2}
  &\mu^m\int_{B_r^{(\theta)}\times\{t_2\}}u\,\dxdt
  +
  \int_{(B_r^{(\theta)}\times\{t_2\})\cap\{\power um \le \power \psi m+\mu^m\}}\mathfrak{b}[\power
  um,\power\psi m+\mu^m]\,\dx \nonumber\\\nonumber
  &\qquad\le
  \mu^m\int_{B_r^{(\theta)}\times\{t_1\}}u\,\dxdt
  +
  c\mu^m\int_{t_1}^{t_2}\int_{\partial B_r^{(\theta)}}(|D\power um| + |F|) \d\mathcal{H}^{n-1}\dt\\\nonumber
  &\qquad\qquad+
  \int_{(B_r^{(\theta)}\times\{t_1\})\cap\{\power um \le \power \psi m+\mu^m\}}\mathfrak{b}[\power
  um,\power\psi m+\mu^m]\,\dx\\\nonumber
  &\qquad\qquad
  +
 	c\mu^m|Q_\rho^{(\theta)}| \bigg[\biint_{Q_\rho^{(\theta)}} |u|^{1+m} \dxdt \bigg]^{\frac{1-m}{2(1+m)}}\\\nonumber
	&\qquad\qquad\qquad\times
	\bigg[\biint_{Q_\rho^{(\theta)}} \big( |u|^{1+m} + |\partial_t \power \psi {m}|^{\frac{1+m}m}\big) \dxdt\bigg]^{\frac{1}{2}}\\
	&\qquad\qquad
	+
	c\iint_{Q_\rho^{(\theta)}}\big(G+\mu^{1+m}
        \big) \dxdt 
\end{align}
for a.e. $r\in [\frac\rho2,\rho]$, where $\mathcal{H}^{n-1}$ denotes
the $(n-1)$-dimensional Hausdorff measure. 
Observing $\mathfrak{b}[\cdot,\cdot]\ge0$,
we discard the second term on the left-hand side, and then divide by $\mu^m|B_r^{(\theta)}|$ on both sides to deduce 
\begin{align}\label{u-(t_2-t_1)}\nonumber
  \<u\>_r(t_2)
  &\le
  \<u\>_r(t_1)
  +
  c\rho^{\frac1m}\theta^{\frac{m(1-m)}{1+m}}\mint_{\Lambda_\rho}\mint_{\partial
    B_r^{(\theta)}}(|D\power um| + |F|) \ \d\mathcal{H}^{n-1}\dt\\\nonumber
  &\qquad+
  \frac1{\mu^m\big|B_r^{(\theta)}\big|}\int_{(B_r^{(\theta)}\times\{t_1\})\cap\{\power um \le \power \psi m+\mu^m\}}\mathfrak{b}[\power
  um,\power\psi m+\mu^m]\,\dx\\\nonumber
  &\qquad
  +
   	c\rho^\frac{1+m}m \bigg[\biint_{Q_\rho^{(\theta)}} |u|^{1+m} \dxdt \bigg]^{\frac{1-m}{2(1+m)}}\\\nonumber
	&\quad\quad\quad\quad\times
	\bigg[\biint_{Q_\rho^{(\theta)}} \big( |u|^{1+m} + |\partial_t \power \psi {m}|^{\frac{1+m}m}\big) \dxdt\bigg]^{\frac{1}{2}}\\
	&\qquad
	+
	\frac{c\rho^\frac{1+m}m}{\mu^m}\biint_{Q_\rho^{(\theta)}}G
        \dxdt +c\mu
\end{align}
with $c=c(n,m,\nu,L)$. In the last step, we also used the assumption
$\rho\le1$. 
Meanwhile, from \eqref{boundary-term}, \eqref{a-b-1}, and
\eqref{on_the_set_psi} we infer 
\begin{align*}
\mathfrak{b}[\power um,\power\psi m+\mu^m]
&\le
c\big| \power u{\frac{1+m}2} - \power{(\psi^m + \mu^m)}{\frac{1+m}{2m}} \big|^2\\
&\le
c(|u|^m+|\psi|^m+\mu^m)^{\frac{1-m}m}|\power um-\power\psi m -\mu^m|^2\\
&\le
c(|u|^{1-m}+\mu^{1-m})\mu^{2m}
\end{align*}
on the set $\{\power um \le \power \psi m+\mu^m\}$, which implies in turn
\begin{align*}
&\frac1{\mu^m\big|B_r^{(\theta)}\big|}\int_{(B_r^{(\theta)}\times\{t_1\})\cap\{\power um \le \power \psi m+\mu^m\}}\mathfrak{b}[\power
  um,\power\psi m+\mu^m]\,\dx\\
&\qquad\le
c\mu^m\bint_{B_r^{(\theta)}\times \{t_1\}} |u|^{1-m} \dx + c\mu \\
&\qquad\le
c\mu^m\left(\bint_{B_r^{(\theta)}\times \{t_1\}} |u|^{1+m} \dx\right)^{\frac{1-m}{1+m}} + c\mu.
\end{align*}
Using the energy estimate~\eqref{energy:a=0} from
Lemma~\ref{lem:energy}, we obtain
\begin{align*}
&\frac1{\mu^m|B_r^{(\theta)}|}\int_{(B_r^{(\theta)}\times\{t_1\})\cap\{\power um \le \power \psi m+\mu^m\}}\mathfrak{b}[\power
  um,\power\psi m+\mu^m]\,\dx\\
&\qquad\le
c\mu^m\bigg[\biint_{Q_\rho^{(\theta)}} \big[|u|^{1+m} +
\rho^{\frac{1-m}m}\theta^{\frac{2m(1-m)}{1+m}} |u|^{2m}\big] \dxdt\bigg]^{\frac{1-m}{1+m}} \\
&\quad\qquad
+
c\mu^m\bigg[\rho^{\frac{1+m}m}\biint_{Q_{\rho}^{(\theta)}}G
\dxdt\bigg]^{\frac{1-m}{1+m}} + c\mu.
\end{align*}
Inserting this into \eqref{u-(t_2-t_1)} we have 
\begin{align}\label{u-(t_2-t_1)-2}\nonumber
&\<u\>_{x_o;r}(t_2)-\<u\>_{x_o;r}(t_1)\\\nonumber
&\qquad\le
c\rho^{\frac1m}\theta^{\frac{m(1-m)}{1+m}}\bint_{\Lambda_\rho}\bint_{\partial B_r^{(\theta)}} \big(|Du^m|+|F|\big)\,\d\mathcal{H}^{n-1}\dt +c\mu\\\nonumber
&\qquad\qquad+
c\mu^m\bigg[\biint_{Q_\rho^{(\theta)}}\big[ |u|^{1+m} 
+
\rho^{\frac{1-m}m}\theta^{\frac{2m(1-m)}{1+m}} |u|^{2m} \big]\dxdt\bigg]^{\frac{1-m}{1+m}}  \\\nonumber
&\qquad\qquad+
c\rho^{\frac{1+m}m}\bigg[\biint_{Q_\rho^{(\theta)}} |u|^{1+m} \dxdt \bigg]^{\frac{1-m}{2(1+m)}} \\\nonumber
&\qquad\qquad\qquad\times
\bigg[\biint_{Q_\rho^{(\theta)}}|u|^{1+m}+ |\partial_t \power \psi m|^{\frac{1+m}m} \dxdt\bigg]^{\frac12} \\
&\qquad\qquad+
  \frac{c\rho^{\frac{1+m}m}}{\mu^m}
  \biint_{Q_\rho^{(\theta)}}
  G\dxdt
+
c\mu^m\bigg[\rho^{\frac{1+m}m}\biint_{Q_{\rho}^{(\theta)}} G
\dxdt\bigg]^{\frac{1-m}{1+m}}
\end{align}
for a.e. $r\in\big[ \frac\rho2, \rho \big]$.

In the same manner we next construct the inequality of the other side, using the comparison function
\begin{equation*}
  \power{v_h}{m}:=\max\{\llbracket
  \power{u}{m}\rrbracket_h+1,\power{\psi}{m}\}.
\end{equation*}
We then argue as above, replacing $\mu^m$ by $-1$. In this case,
however, the estimation is much simpler. This is due to the fact that 
the set $\{\power um\le\power \psi m-1\}$ is empty by
the obstacle constraint, and therefore all terms that arise from  
integrals over $\{\power um\le\power \psi m-1\}$ can be
neglected in this case. 
In this way, we arrive at the estimate
\begin{align}\label{u-(t_1-t_2)}\nonumber
\<u\>_{x_o;r}(t_1)-\<u\>_{x_o;r}(t_2)
  &\le
  c\rho^{\frac1m}\theta^{\frac{m(1-m)}{1+m}}\bint_{\Lambda_\rho}\bint_{\partial B_r^{(\theta)}}(|D\power um| + |F|)\,\d\mathcal{H}^{n-1}\dt \\
  &\qquad+
 \rho^{\frac{1+m}m} \biint_{Q_\rho^{(\theta)}} |g| \dxdt
\end{align}
for a.e. $r\in[\frac\rho2, \rho]$.

Meanwhile, the mean value theorem for integrals implies
\begin{align}
\bint_{\Lambda}\bint_{\partial B_{\hat{\rho}}^{(\theta)}} |Du^m|+|F|\,\d\mathcal{H}^{n-1}\dt 
\le
\biint_{Q_\rho^{(\theta)}} |Du^m|+|F|\dxdt
\end{align}
for some $\hat{\rho}\in\big[ \frac\rho2 , \rho \big]$.
We then apply \eqref{u-(t_2-t_1)-2} and \eqref{u-(t_1-t_2)} with $r=\hat{\rho}$ and combine both inequalities to obtain the desired estimate.
\end{proof}

The next step is the proof of a Sobolev-Poincar\'e inequality on
sub-intrinsic cylinders. This can be established in two cases, which
can be interpreted as the singular and the non-singular case.
The first one is characterized by the fact that the mean value
of $u$ is bounded in terms of the oscillation of $u$ in the sense
of \eqref{intrinsic_2} below. It turns out that in the other case, it
is possible to construct cylinders that are even intrinsic in the
sense that \eqref{sub-intrinsic-1} and \eqref{intrinsic_1} are both
valid at the same time. Therefore, the following
result will be applicable on any suitably constructed cylinder.

\begin{lemma}\label{lem:sobolev-poin}
Let $\frac{(n-2)_+}{n+2}<m\le1$ and $u\in K_\psi$ be a local weak
solution to the variational inequality \eqref{vari-ineq} under assumptions \eqref{obs-func}, \eqref{inhom-func}, and \eqref{A-struc}. 
For $0<\rho\le1$, $\theta>0$, and $z_o=(x_o,t_o)$, 
we consider parabolic cylinders $Q_\rho^{(\theta)}(z_o)\subset\Omega_T$
which satisfy a sub-intrinsic coupling in the sense that
\begin{align}\label{sub-intrinsic-1}
\biint_{Q_{\rho}^{(\theta)}(z_o)} \frac{|u|^{1+m}}{\rho^{\frac{1+m}{m}}} \dxdt
\le
B\theta^{2m}
\end{align}
for some constant $B\ge1$.
Furthermore, we assume that either
\begin{align}\label{intrinsic_1}
\theta^{2m} \le B\biint_{Q_{\rho}^{(\theta)}(z_o)} \frac{|u|^{1+m}}{\rho^{\frac{1+m}{m}}} \dxdt
\end{align}
or
\begin{align}\label{intrinsic_2}
\theta^{2m} \le B\biint_{Q_{\rho}^{(\theta)}(z_o)} \big[ |D\power um|^2+G+|u|^{1+m}\big] \dxdt
\end{align}
hold true, where we used the abbreviation
$$G:=|F|^2+|g|^{1+m} +|D\power\psi m|^2+|\partial_t\power\psi m|^{\frac{1+m}m}.$$
Then, for any $\varepsilon\in(0,1)$ we have the estimate
\begin{align}\label{desired-est}
\nonumber
&\biint_{Q_\rho^{(\theta)}(z_o)} \frac{\big| \power u{\frac{1+m}2} -\big(\power u{\frac{1+m}2}\big)_\rho^{(\theta)} \big|^2}{\rho^{\frac{1+m}m}} \dxdt\\\nonumber
&\le
\epsilon\bigg[ \sup_{t\in\Lambda_\rho(t_o)} \bint_{B_\rho^{(\theta)}(x_o)} \frac{\big| \power u{\frac{1+m}2} -\big(\power u{\frac{1+m}2}\big)_\rho^{(\theta)} \big|^2}{\rho^{\frac{1+m}m}} \dx
+
 \biint_{Q_\rho^{(\theta)}(z_o)} |D\power um|^2 \dxdt \bigg]\\
&\qquad+
\frac{ c}{\epsilon^{\frac{2}{n}}} \bigg[\biint_{Q_\rho^{(\theta)}(z_o)} |D\power um|^{2q_o}\dxdt\bigg]^{\frac1{q_o}}
+
\frac{c}{\epsilon^{\frac{2}{n}}}\biint_{Q_\rho^{(\theta)}(z_o)} \big(G+|u|^{1+m}\big)\dxdt,
\end{align}
with a constant $c=c(n,m,\nu,L,B)>0$ and the exponent
$$q_o:=\max\Bigg\{\frac12, \frac{n(1+m)}{2(nm+1+m)}\Bigg\}<1.$$
\end{lemma}
\begin{proof}
Again, we omit the reference to the center point $z_o=(x_o,t_o)$
throughout the proof. For the proof we follow the line of argument
from \cite[Lemma 5.1]{BDS-singular-higher-int}.
Let us start the proof by introducing the radius $\hat{\rho} \in\big[
\frac\rho2 , \rho \big]$ that appeared in Lemma~\ref{lem:time-diff}.
Then we use Lemma~\ref{lem:mvint} to estimate
\begin{align}\label{Estimate}\nonumber
\biint_{Q_\rho^{(\theta)}}\frac{\big|\power u {\frac{1+m}2}- \big(\power u {\frac{1+m}2}\big)_\rho^{(\theta)}\big|^2}{\rho^{\frac{1+m}m}}\dxdt
&\le
c\,\biint_{Q_\rho^{(\theta)}}\frac{\big|\power u {\frac{1+m}2}- \power{\big[(u)_{\hat\rho}^{(\theta)}\big]} {\frac{1+m}2}\big|^2}{\rho^{\frac{1+m}m}}\dxdt\\
&\le
c({\rm{I+II}}),
\end{align}
where
\begin{align*}
{\rm{I}} :=
\biint_{Q_\rho^{(\theta)}}\frac{\big| \power u{\frac{1+m}2} - \power{\big[\<u\>_{\hat{\rho}}^{(\theta)}(t)\big]}{\frac{1+m}2} \big|^2 }{\rho^{\frac{1+m}m}} \dxdt
\end{align*}
and
\begin{align*}
{\rm{II}} :=
\bint_{\Lambda_\rho}\frac{\big|  \power{\big[\<u\>_{\hat{\rho}}^{(\theta)}(t)\big]}{\frac{1+m}2} - \power{\big[ (u)_{\hat{\rho}}^{(\theta)} \big]}{\frac{1+m}2} \big|^2 }{\rho^{\frac{1+m}m}} dt.
\end{align*}
To estimate {\rm{I}}, we first use Young's inequality and
Lemma~\ref{lem:mvint}, with the result
\begin{align}\label{est-I-a}
{\rm{I}} 
&\le
\sup_{t\in \Lambda_\rho}\bigg[\bint_{B_\rho^{(\theta)} }\frac{\big|\power u{\frac{1+m}2} - \power{\big[\<u\>_{\hat{\rho}}^{(\theta)}(t)\big]}{\frac{1+m}2}\big|^2}{\rho^\frac{1+m}m}  \dx\bigg]^\frac2{n+2} \nonumber\\
&\qquad\times
\mint_{\Lambda_{\rho}}\bigg[\mint_{B_\rho^{(\theta)}} \frac{\big|\power u {\frac{1+m}2} - \power{\big[\< u\>_{\hat{\rho}}^{(\theta)}(t)\big]}{\frac{1+m}{2}} \big|^2}{\rho^\frac{1+m}m} \dx\bigg]^\frac{n}{n+2} \dt \nonumber\\
&\le
\epsilon_1\sup_{t\in \Lambda_\rho}\bint_{B_\rho^{(\theta)} }\frac{\big|\power u{\frac{1+m}2} - \big(\power u{\frac{1+m}m}\big)_\rho^{(\theta)}\big|^2 }{\rho^{\frac{1+m}m}} \dx
+
\frac{c}{\epsilon_1^{\frac{2}{n}}\rho^\frac{1+m}m} {\rm{I}_1}^\frac{n+2}n
\end{align}
for any $\epsilon_1\in(0,1)$, where

\begin{equation}\label{def-I1}
{\rm{I}_1}
:=
\mint_{\Lambda_{\rho}}\bigg[\mint_{B_\rho^{(\theta)}} \big|\power
u{\frac{1+m}2} - \power{\big[\<\power
  um\>_\rho^{(\theta)}(t)\big]}{\frac{1+m}{2m}} \big|^2
\dx\bigg]^\frac{n}{n+2} \dt.
\end{equation}
For the estimate of $\rm{I}_1$ in the case of $m<1$, we deduce from
\eqref{a-b-1} with $\alpha=\frac{1+m}{2m}$ that
\begin{align*}
\big|\power u{\frac{1+m}2} - \power{\big[\<\power um\>_\rho^{(\theta)}(t)\big]}{\frac{1+m}{2m}} \big|
\le
c\big(\big|\power um\big|+\big|\<\power um\>_\rho^{(\theta)}(t)\big|\big)^\frac{1-m}{2m}
\big|\power um-\<\power um\>_\rho^{(\theta)}(t)\big|,
\end{align*}
holds true, which implies
\begin{align}
{\rm{I}_1}
&\le
c\,\bint_{\Lambda_\rho}\bigg[\bint_{B_\rho^{(\theta)}}\big(\big|\power um\big|+\big|\<\power um\>_\rho^{(\theta)}(t)\big|\big)^\frac{1-m}{m}\big|\power um-\<\power um\>_\rho^{(\theta)}(t)\big|^2
\dx\bigg]^\frac{n}{n+2}\dt \nonumber\\
&\le
c\,\bint_{\Lambda_\rho}\bigg[ \bint_{B_\rho^{(\theta)}}|u|^{1+m}\dx\bigg]^\frac{n(1-m)}{(n+2)(1+m)}
\bigg[ \bint_{B_\rho^{(\theta)}} \big|\power um-\<\power um\>_\rho^{(\theta)}(t)\big|^\frac{1+m}m \dx \bigg]^\frac{2nm}{(n+2)(1+m)} \dt \nonumber\\
&\le
c\bigg[\biint_{Q_\rho^{(\theta)}}|u|^{1+m}\dxdt\bigg]^\frac{n(1-m)}{(n+2)(1+m)}\nonumber\\
&\qquad\times
\Bigg[\bint_{\Lambda_\rho}\bigg[ \bint_{B_\rho^{(\theta)}} \big|\power um-\<\power um\>_\rho^{(\theta)}(t)\big|^\frac{1+m}m \dx \bigg]^\frac{nm}{nm+1+m} \dt\Bigg]^\frac{2(nm+1+m)}{(n+2)(1+m)}.\nonumber
\end{align}
Here, we used H\"older's inequality in space with exponents
$\frac{1+m}{1-m}$ and $\frac{1+m}{2m}$ and then in time with exponents
$\frac{(n+2)(1+m)}{n(1-m)}$ and $\frac{(n+2)(1+m)}{2(nm+1+m)}$. 
Using \eqref{sub-intrinsic-1} and Sobolev's inequality on the time slices, this leads to
\begin{align}
{\rm{I}_1}
&\le
c\big(\theta^{2m}\rho^\frac{1+m}m\big)^\frac{n(1-m)}{(n+2)(1+m)}\big(\theta^\frac{m(m-1)}{1+m}\rho\big)^\frac{2n}{n+2}
\bigg[\biint_{Q_\rho^{(\theta)}}|D\power um|^{2q_o}\dxdt\bigg]^\frac{n}{q_o(n+2)} \nonumber\\
&=
c\Bigg[\rho^\frac{1+m}{m}
\bigg[\biint_{Q_\rho^{(\theta)}}|D\power um|^{2q_o}\dxdt\bigg]^\frac{1}{q_o}\Bigg]^\frac{n}{n+2}	
\end{align}
for some $c=c(n,m,B)>0$, where
$$q_o:= \max\bigg\{ \frac12, \frac{n(1+m)}{2(nm+1+m)}\bigg\}.$$
In the case $m=1$, the same inequality follows more directly by
applying Sobolev's inequality slicewise to the integral
in~\eqref{def-I1}. 
Now we insert the last inequality into \eqref{est-I-a}, in order to get
\begin{align}\label{est-I-aa}
{\rm{I}}
&\le
\epsilon_1\sup_{t\in \Lambda_\rho}\bint_{B_\rho^{(\theta)} }\frac{\big|\power u{\frac{1+m}2} - \big(\power u{\frac{1+m}2}\big)_\rho^{(\theta)}\big|^2 }{\rho^{\frac{1+m}m}} \dx 
+
\frac{c}{\epsilon_1^\frac 2n}
\bigg[\biint_{Q_\rho^{(\theta)}}|D\power um|^{2q_o}\dxdt\bigg]^\frac{1}{q_o}.
\end{align}
For the estimate of {\rm{II}}, we first consider the case that
the super-intrinsic coupling \eqref{intrinsic_1} is satisfied and that $m<1$.
In this case, condition \eqref{intrinsic_1} implies
\begin{align}\label{intrinsic-1-theta2m}
\theta^{2m}
&\le
2B\biint_{Q_{\rho}^{(\theta)}}\frac{\big|\power u{\frac{1+m}2} - \power{\big[ (u)_{\hat{\rho}}^{(\theta)}\big]}{\frac{1+m}2}\big|^2}{\rho^{\frac{1+m}m}} \dxdt
+
\frac{2B\big|(u)_{\hat{\rho}}^{(\theta)} \big|^{1+m}}{\rho^{\frac{1+m}m}}.
\end{align}
Using \eqref{intrinsic-1-theta2m} and \eqref{a-b-2}, we infer 
\begin{align}\label{est-II-a}
{\rm{II}}
\le
\frac{c\theta^{\frac{2m(1-m)}{1+m}}}{\theta^{\frac{2m(1-m)}{1+m}}} \bint_{\Lambda_\rho}\frac{\big|  \power{\big[\<u\>_{\hat{\rho}}^{(\theta)}(t)\big]}{\frac{1+m}2} - \power{\big[ (u)_{\hat{\rho}}^{(\theta)} \big]}{\frac{1+m}2} \big|^2 }{\rho^{\frac{1+m}m} } \dt 
\le
c({\rm{II}_1} + {\rm{II}_2}),
\end{align}
where
\begin{align*}
{\rm{II}_1}
&:=
\frac1{\big(\theta^{2m}\big)^\frac{(1-m)^2}{2(1+m)}}
\left[\biint_{Q_{\rho}^{(\theta)}}\frac{\big|\power u{\frac{1+m}2} - \power{\big[ (u)_{\hat{\rho}}^{(\theta)}\big]}{\frac{1+m}2}\big|^2}{\rho^{\frac{1+m}m}} \dxdt\right]^{\frac{1-m}{1+m}}\\
&\qquad\times
\bint_{\Lambda_\rho}\frac{\big| \<u\>_{\hat{\rho}}^{(\theta)}(t) - (u)_{\hat{\rho}}^{(\theta)} \big|^{1+m}}{\rho^{\frac{1+m}m}\theta^{m(1-m)} } \dt
\end{align*}
and 
\begin{align*}
{\rm{II}_2}
:= 
\bint_{\Lambda_\rho}\frac{\big|(u)_{\hat{\rho}}^{(\theta)} \big|^{1-m}
  \big|
  \power{\big[\<u\>_{\hat{\rho}}^{(\theta)}(t)\big]}{\frac{1+m}2} -
  \power{\big[ (u)_{\hat{\rho}}^{(\theta)} \big]}{\frac{1+m}2} \big|^2
  }{\rho^{\frac{2}m} \theta^{\frac{2m(1-m)}{1+m}}} \dt .
\end{align*}
Noting that $\frac{1-m}{1+m}-\frac{(1-m)^2}{2(1+m)} = \frac{1-m}2$, we
use the sub-intrinsic coupling \eqref{sub-intrinsic-1} and H\"older's inequality to estimate
${\rm{II}_1}$ as follows.
\begin{align*}
{\rm{II}_1}
\le
\bigg[ \biint_{Q_{\rho}^{(\theta)}}\frac{\big|\power u{\frac{1+m}2} - \power{\big[ (u)_{\hat{\rho}}^{(\theta)}\big]}{\frac{1+m}2}\big|^2}{\rho^{\frac{1+m}m}} \dxdt \bigg]^{\frac{1-m}2}
\bigg[\bint_{\Lambda_\rho}\frac{\big| \<u\>_{\hat{\rho}}^{(\theta)}(t) - (u)_{\hat{\rho}}^{(\theta)} \big|^{2}}{\rho^{\frac{2}m} \theta^{\frac{2m(1-m)}{1+m}}} \dt\bigg]^{\frac{1+m}2}. 
\end{align*}
According to Young's inequality and Lemma~\ref{lem:mvint}, this
implies for any $\epsilon_o>0$ that 
\begin{align}\label{est-II1-a}\nonumber
{\rm{II}_1}
&\le
\epsilon_o \biint_{Q_{\rho}^{(\theta)}}\frac{\big|\power u{\frac{1+m}2} - \power{\big[ (u)_{\hat{\rho}}^{(\theta)}\big]}{\frac{1+m}2}\big|^2}{\rho^{\frac{1+m}m}} \dxdt 
+
\frac{c}{\epsilon_o^{\frac{1-m}{1+m}}}\bint_{\Lambda_\rho}\frac{\big| \<u\>_{\hat{\rho}}^{(\theta)}(t) - (u)_{\hat{\rho}}^{(\theta)} \big|^{2}}{\rho^{\frac{2}m} \theta^{\frac{2m(1-m)}{1+m}}} \dt \\\nonumber
&\le
c\epsilon_o \biint_{Q_{\rho}^{(\theta)}}\frac{\big|\power u{\frac{1+m}2} - \big(\power{u}{\frac{1+m}2}\big)_{\rho}^{(\theta)}\big|^2}{\rho^{\frac{1+m}m}} \dxdt \\
&\qquad+
\frac{c}{\epsilon_o^{\frac{1-m}{1+m}}}\bint_{\Lambda_\rho}\bint_{\Lambda_\rho}\frac{\big| \<u\>_{\hat{\rho}}^{(\theta)}(t) - \<u\>_{\hat{\rho}}^{(\theta)}(\tau) \big|^{2}}{\rho^{\frac{2}m} \theta^{\frac{2m(1-m)}{1+m}}} \dt \d\tau.
\end{align}
For the estimate of ${\rm{II}_2}$ we apply Lemma~\ref{lem:a-b}, which
yields 
\begin{align}
{\rm{II}_2}\label{est-II2-a}\nonumber
&\le
c\,\bint_{\Lambda_\rho}\frac{  \big| \<u\>_{\hat{\rho}}^{(\theta)}(t) - (u)_{\hat{\rho}}^{(\theta)} \big|^{2}}{\rho^{\frac{2}m} \theta^{\frac{2m(1-m)}{1+m}}} \dt\\
&\le
c\,\bint_{\Lambda_\rho}\bint_{\Lambda_\rho}\frac{\big| \<u\>_{\hat{\rho}}^{(\theta)}(t) - \<u\>_{\hat{\rho}}^{(\theta)}(\tau) \big|^{2}}{\rho^{\frac{2}m} \theta^{\frac{2m(1-m)}{1+m}}} \dt \d\tau.
\end{align}
Combining \eqref{est-II-a}-\eqref{est-II2-a} and selecting
$\epsilon_o=\epsilon_o(n,m,\nu,L,B)>0$ so small that $c\epsilon_o\le\frac12$, we infer
\begin{align}\label{est-II1}
\nonumber
{\rm{II}}
&\le
\frac12 \biint_{Q_{\rho}^{(\theta)}}\frac{\big|\power u{\frac{1+m}2} - \big( \power u{\frac{1+m}2}\big)_{\rho}^{(\theta)}\big|^2}{\rho^{\frac{1+m}m}} \dxdt \\
&\qquad+
c\,\bint_{\Lambda_\rho}\bint_{\Lambda_\rho}\frac{\big| \<u\>_{\hat{\rho}}^{(\theta)}(t) - \<u\>_{\hat{\rho}}^{(\theta)}(\tau) \big|^{2}}{\rho^{\frac{2}m} \theta^{\frac{2m(1-m)}{1+m}}} \dt \d\tau.
\end{align}
In order to estimate the last integrand, we first apply Lemma~\ref{lem:time-diff} and make use of
the sub-intrinsic coupling \eqref{sub-intrinsic-1} to estimate the
right-hand side further. This leads us to the estimate   
\begin{align*}
\nonumber
&\frac{\big|\<u\>_{\hat{\rho}}^{(\theta)}(t)-\<u\>_{\hat{\rho}}^{(\theta)}(\tau)\big|^2}{\rho^{\frac2m}\theta^{\frac{2m(1-m)}{1+m}}}\\\nonumber
&\quad\le
c\bigg[ \biint_{Q_\rho^{(\theta)}} |Du^m|^{2q_o}\dxdt\bigg]^{\frac1{q_o}}
+ 
c\biint_{Q_\rho^{(\theta)}} |F|^2\dxdt\\\nonumber
&\quad\quad+
\frac{c\mu^2}{\rho^{\frac2m}\theta^{\frac{2m(1-m)}{1+m}}}
+
\frac{c\mu^{2m}\theta^{\frac{2m(1-m)}{1+m}}}{\rho^2} \\\nonumber
&\quad\quad+
c\rho^{\frac{1+m}m}
\biint_{Q_\rho^{(\theta)}}\big(|u|^{1+m}+ |\partial_t \power \psi m|^{\frac{1+m}{m}}\big) \dxdt \\
&\quad\quad
+
  \frac{c\rho^{2}}{\mu^{2m}\theta^{\frac{2m(1-m)}{1+m}}}
\bigg[  \biint_{Q_\rho^{(\theta)}}G\dx\dt\bigg]^2
  +
\frac{c\mu^{2m}}{\rho^2\theta^{\frac{2m(1-m)}{1+m}}}\bigg[\biint_{Q_\rho^{(\theta)}}G
\dxdt\bigg]^{\frac{2-2m}{1+m}}
\end{align*}
for any $t,\tau\in\Lambda_\rho$, with a constant $c=c(n,m,\nu,L,B)$.
We use this estimate in \eqref{est-II1}, 
with the result
\begin{align}\label{est-II2}
\nonumber
{\rm{II}}
&\le
\frac12 \biint_{Q_{\rho}^{(\theta)}}\frac{\big|\power u{\frac{1+m}2} - \big( \power u{\frac{1+m}2}\big)_{\rho}^{(\theta)}\big|^2}{\rho^{\frac{1+m}m}} \dxdt \\\nonumber
&\qquad+
c\bigg[ \biint_{Q_\rho^{(\theta)}} |Du^m|^{2q_o}\dxdt\bigg]^{\frac1{q_o}}
+ 
c\biint_{Q_\rho^{(\theta)}} |F|^2\dxdt\\\nonumber
&\quad\quad+
\frac{c\mu^2}{\rho^{\frac2m}\theta^{\frac{2m(1-m)}{1+m}}}
+
\frac{c\mu^{2m}\theta^{\frac{2m(1-m)}{1+m}}}{\rho^2} \\\nonumber
&\quad\quad+
c\rho^{\frac{1+m}m}
\biint_{Q_\rho^{(\theta)}}\big(|u|^{1+m}+ |\partial_t \power \psi m|^{\frac{1+m}{m}}\big) \dxdt \\
&\quad\quad
+
  \frac{c\rho^{2}}{\mu^{2m}\theta^{\frac{2m(1-m)}{1+m}}}
\bigg[  \biint_{Q_\rho^{(\theta)}}G\dx\dt\bigg]^2
  +
\frac{c\mu^{2m}}{\rho^2\theta^{\frac{2m(1-m)}{1+m}}}\bigg[\biint_{Q_\rho^{(\theta)}}G
\dxdt\bigg]^{\frac{2-2m}{1+m}}.
\end{align}
With the choice 
\begin{align*}
\mu^{2m}:=\frac{\rho^2}{\theta^{\frac{2m(1-m)}{1+m}} }\biint_{Q_\rho^{(\theta)}} \big(G+ \delta\big)\dxdt>0
\end{align*}
for some $\delta>0$, the previous inequality yields the bound
\begin{align}\label{est-II-aa}
\nonumber
{\rm{II}}
&\le
\frac12 \biint_{Q_{\rho}^{(\theta)}}\frac{\big|\power u{\frac{1+m}2} - \big( \power u{\frac{1+m}2}\big)_{\rho}^{(\theta)}\big|^2}{\rho^{\frac{1+m}m}} \dxdt \\\nonumber
&\qquad+
c\bigg[ \biint_{Q_\rho^{(\theta)}} |Du^m|^{2q_o}\dxdt\bigg]^{\frac1{q_o}}\\\nonumber
&\qquad
+
\frac c{\theta^{2(1-m)}}\bigg[ \biint_{Q_\rho^{(\theta)}} \big(G+ \delta\big) \dxdt\bigg]^{\frac1m}
+
c\,\biint_{Q_\rho^{(\theta)}}\big(G+\delta)\dxdt
\\\nonumber
&\qquad+
c\rho^{\frac{1+m}m}
\biint_{Q_\rho^{(\theta)}}|u|^{1+m}+ |\partial_t \power \psi m|^{\frac{1+m}{m}} \dxdt \\
&\qquad+
  \frac{c}{\theta^{\frac{4m(1-m)}{1+m}}}\bigg[
\biint_{Q_\rho^{(\theta)}}\big(G+\delta\big)
\dxdt\bigg]^{\frac{3-m}{1+m}}.
\end{align}
We now plug \eqref{est-I-aa} and \eqref{est-II-aa} into \eqref{Estimate}, and then let $\delta\downarrow0$ to deduce
\begin{align}\label{pre-poincare-est}
\nonumber
&\biint_{Q_{\rho}^{(\theta)}}\frac{\big|\power u{\frac{1+m}2} - \big( \power u{\frac{1+m}2}\big)_{\rho}^{(\theta)}\big|^2}{\rho^{\frac{1+m}m}} \dxdt \\\nonumber
&\qquad\le
\frac12 \biint_{Q_{\rho}^{(\theta)}}\frac{\big|\power u{\frac{1+m}2} -
  \big( \power
  u{\frac{1+m}2}\big)_{\rho}^{(\theta)}\big|^2}{\rho^{\frac{1+m}m}}
\dxdt\\\nonumber
&\qquad\qquad+
\epsilon_1\sup_{t_o\in\Lambda_\rho}\bint_{B_{\rho}^{(\theta)}}\frac{\big|\power u{\frac{1+m}2} - \big( \power u{\frac{1+m}2}\big)_{\rho}^{(\theta)}\big|^2}{\rho^{\frac{1+m}m}} \dx
 \\\nonumber
&\qquad\qquad+
\frac{c}{\epsilon_1^{\frac2n}}\bigg[ \biint_{Q_\rho^{(\theta)}}
|Du^m|^{2q_o}\dxdt\bigg]^{\frac1{q_o}}\\\nonumber
&\qquad\qquad+
\frac c{\theta^{2(1-m)}}\bigg[ \biint_{Q_\rho^{(\theta)}} G \dxdt\bigg]^{\frac1m}
+
c\,\biint_{Q_\rho^{(\theta)}}G\dxdt
\\\nonumber
&\qquad\qquad+
c\rho^{\frac{1+m}m}
\biint_{Q_\rho^{(\theta)}}\big(|u|^{1+m}+ |\partial_t \power \psi m|^{\frac{1+m}{m}}\big) \dxdt \\
&\qquad\qquad+
  \frac{c}{\theta^{\frac{4m(1-m)}{1+m}}}\bigg[
\biint_{Q_\rho^{(\theta)}}G
\dxdt\bigg]^{\frac{3-m}{1+m}}.
\end{align}
We then re-absorb the first term on the right-hand side into the left-hand side.
Next, we observe that the sub-intrinsic coupling
\eqref{sub-intrinsic-1} implies 
\begin{align*}
E:=\biint_{Q_{\rho}^{(\theta)}}\frac{\big|\power u{\frac{1+m}2} - \big( \power u{\frac{1+m}2}\big)_{\rho}^{(\theta)}\big|^2}{\rho^{\frac{1+m}m}} \dxdt
\le
c\theta^{2m}.
\end{align*}
We multiply inequality~\eqref{pre-poincare-est} by $E^{\frac{1-m}m}$
and apply Young's inequality. 
In this way, we obtain for any $\epsilon_o\in(0,1)$
\begin{align}
\nonumber
&\bigg[ \biint_{Q_{\rho}^{(\theta)}}\frac{\big|\power u{\frac{1+m}2} - \big( \power u{\frac{1+m}2}\big)_{\rho}^{(\theta)}\big|^2}{\rho^{\frac{1+m}m}} \dxdt \bigg]^{\frac1m} \\\nonumber
&\qquad\le
\frac{c\epsilon_1}{\epsilon_o^{\frac{1-m}m}}\bigg[ \sup_{t\in\Lambda_\rho}\bint_{B_{\rho}^{(\theta)}}\frac{\big|\power u{\frac{1+m}2} - \big( \power u{\frac{1+m}2}\big)_{\rho}^{(\theta)}\big|^2}{\rho^{\frac{1+m}m}} \dx\bigg]^{\frac1m}\\\nonumber
&\qquad\qquad+
5
\epsilon_o\bigg[  \biint_{Q_{\rho}^{(\theta)}}\frac{\big|\power u{\frac{1+m}2} - \big( \power u{\frac{1+m}2}\big)_{\rho}^{(\theta)}\big|^2}{\rho^{\frac{1+m}m}} \dxdt \bigg]^{\frac1m}\\\nonumber
&\qquad\qquad+
\frac{c}{\epsilon_1^{\frac2n}\epsilon_o^{\frac{1-m}m}}\bigg[ \biint_{Q_\rho^{(\theta)}} |Du^m|^{2q_o}\dxdt\bigg]^{\frac1{q_om}}\\
&\qquad\qquad+
\frac{c}{\epsilon_o^{\frac{1-m}m}}\bigg[ \biint_{Q_\rho^{(\theta)}}\big(|u|^{1+m}+ G\big) \dxdt\bigg]^{\frac1m}.
\end{align}
Here, we also used the fact $\rho\le1$.
Choosing $\epsilon_o=\frac1{10}$, we can re-absorb the second term on
the right-hand side into the left-hand side. In this way, we obtain 
\begin{align}
\nonumber
&\biint_{Q_{\rho}^{(\theta)}}\frac{\big|\power u{\frac{1+m}2} - \big( \power u{\frac{1+m}2}\big)_{\rho}^{(\theta)}\big|^2}{\rho^{\frac{1+m}m}} \dxdt \\\nonumber
&\qquad\le
c\epsilon_1^m \sup_{t_o\in\Lambda_\rho}\bint_{B_{\rho}^{(\theta)}}\frac{\big|\power u{\frac{1+m}2} - \big( \power u{\frac{1+m}2}\big)_{\rho}^{(\theta)}\big|^2}{\rho^{\frac{1+m}m}} \dx\\
&\qquad\qquad+
\frac{c}{\epsilon_1^{\frac{2m}n}}\bigg[ \biint_{Q_\rho^{(\theta)}} |Du^m|^{2q_o}\dxdt\bigg]^{\frac1{q_o}}
+
 c\,\biint_{Q_\rho^{(\theta)}} \big( |u|^{1+m} +G \big) \dxdt,
\end{align}
which implies the desired estimate if we choose $\epsilon_1$ such that $c\epsilon_1^m=\epsilon$.

It remains to consider the case that
either $m=1$ or \eqref{intrinsic_2} is satisfied.
In this situation
we use in turn \eqref{a-b-2}, Lemma~\ref{lem:time-diff}, and
the sub-intrinsic coupling \eqref{sub-intrinsic-1} for the estimate of ${\rm{II}}$ from
\eqref{Estimate}, with the result 
\begin{align}
\nonumber
{\rm{II}}
&\le
c\,  \bint_{\Lambda_\rho}\frac{\big| \<u\>_{\hat{\rho}}^{(\theta)}(t) - (u)_{\hat{\rho}}^{(\theta)} \big|^{1+m}}{\rho^{\frac{1+m}m} } \dt\\\nonumber
&\le
c\,  \bint_{\Lambda_\rho}\bint_{\Lambda_\rho}\frac{\big| \<u\>_{\hat{\rho}}^{(\theta)}(t) - \<u\>_{\hat{\rho}}^{(\theta)}(\tau) \big|^{1+m}}{\rho^{\frac{1+m}m} } \dt\d\tau\\\nonumber
&\le
c\theta^{m(1-m)}\bigg[ \biint_{Q_\rho^{(\theta)}} \big(|D\power um|+|F|\big) \dxdt\bigg]^{1+m} \\\nonumber
&\qquad+
\frac{c\mu^{1+m}}{\rho^{\frac{1+m}m}}
+
\frac{c\mu^{m(1+m)}\theta^{2m(1-m)}}{\rho^{1+m}} \\\nonumber
&\qquad+
c\theta^{m(1-m)}
\bigg[\rho^\frac{1+m}m \biint_{Q_\rho^{(\theta)}}\big(|u|^{1+m}+
|\partial_t \power \psi m|^{\frac{1+m}{m}}
\big)\dxdt\bigg]^{\frac{1+m}2} \\
&\qquad+ 
\frac{c\mu^{m(1+m)}}{\rho^{1+m}}
\bigg[\biint_{Q_\rho^{(\theta)}} G\dxdt\bigg]^{1-m}
+
  \frac{c\rho^{1+m}}{\mu^{m(1+m)}}
  \bigg[ \biint_{Q_\rho^{(\theta)}}G\dxdt \bigg]^{1+m}.
\end{align}
Next, we use Young's and H\"older's inequalities to
estimate the right-hand side further and obtain 
\begin{align}\label{ii2}
\nonumber
{\rm{II}}
&\le
c\epsilon_o(1-m) \theta^{2m}
+
 \frac{c}{\epsilon_o^{\frac{1-m}{1+m}}} \bigg[\biint_{Q_\rho^{(\theta)}} |D\power um|^{2q_o}\dxdt\bigg]^{\frac1{q_o}}
+
\frac{c\mu^{1+m}}{\epsilon_o^{\frac{1-m}m}\rho^{\frac{1+m}m}}  \\
&\qquad+
\frac{c}{\epsilon_o^{\frac{1-m}{1+m}}}\biint_{Q_\rho^{(\theta)}}
\big(G + |u|^{1+m}\big) \dxdt 
+
 \frac{c\rho^{1+m}}{\mu^{m(1+m)}}
  \bigg[ \biint_{Q_\rho^{(\theta)}}
  G\dx\dt \bigg]^{1+m}
\end{align}
for any $\epsilon_o\in(0,1)$, in both of the cases $m<1$ and $m=1$. Here, we used the fact $\rho\le1$ to
simplify the right-hand side.
We point out that the parameter $\theta$ does only appear in the case
$m<1$. This is the reason why we only need to assume the upper bound
\eqref{intrinsic_2} for $\theta^{2m}$ in the latter case.
Using \eqref{intrinsic_2} in the case $m<1$, we deduce
from~\eqref{ii2} that 
\begin{align}\label{ii3}
\nonumber
{\rm{II}}
&\le
c\epsilon_o \biint_{Q_\rho^{(\theta)}} |D\power um|^2 \dxdt
+
 \frac{c}{\epsilon_o^{\frac{1-m}{1+m}}} \bigg[\biint_{Q_\rho^{(\theta)}} |D\power um|^{2q_o}\dxdt\bigg]^{\frac1{q_o}}
+
\frac{c\mu^{1+m}}{\epsilon_o^{\frac{1-m}m}\rho^{\frac{1+m}m}}  \\
&\qquad+
\frac{c}{\epsilon_o^{\frac{1-m}{1+m}}}\biint_{Q_\rho^{(\theta)}}
\big(G + |u|^{1+m}\big) \dxdt 
+
 \frac{c\rho^{1+m}}{\mu^{m(1+m)}}
  \bigg[ \biint_{Q_\rho^{(\theta)}}
  G\dx\dt \bigg]^{1+m}
\end{align}
for any $\epsilon_o\in(0,1)$.
This estimate also holds in the case $m=1$,
even if \eqref{intrinsic_2} is not valid. 
%}
In the preceding estimate, we choose the parameter 
\begin{align*}
  \mu^{1+m} :=\eps_o^{\frac{1-m}{m(1+m)}}
\rho^{\frac{1+m}m}\biint_{Q_\rho^{(\theta)}} \big(G+\delta\big) \dxdt >0
\end{align*}
for any $\delta>0$, and then let $\delta\downarrow0$. In this way, we deduce 
\begin{align}
\nonumber
{\rm{II}}
&\le
c\epsilon_o \biint_{Q_\rho^{(\theta)}} |D\power um|^2 \dxdt
+
\frac{ c}{\epsilon_o^{\frac{1-m}{1+m}}} \bigg[\biint_{Q_\rho^{(\theta)}} |D\power um|^{2q_o}\dxdt\bigg]^{\frac1{q_o}}\\\nonumber
&\qquad+
\frac{ c}{\epsilon_o^{\frac{1-m}{1+m}}}\biint_{Q_\rho^{(\theta)}} \big(G+|u|^{1+m}\big)\dxdt.
\end{align}
We finally combine this with the estimate \eqref{est-I-aa} for
the term $\mathrm{I}$. In view of \eqref{Estimate}, we arrive at the bound
\begin{align}
\nonumber
&\biint_{Q_\rho^{(\theta)}}\frac{\big|\power u {\frac{1+m}2}- \big(\power u {\frac{1+m}2}\big)_\rho^{(\theta)}\big|^2}{\rho^{\frac{1+m}m}}\dxdt \\\nonumber
&\qquad\le
c\epsilon_1\sup_{t\in \Lambda_\rho}\bint_{B_\rho^{(\theta)} }\frac{\big|\power u{\frac{1+m}2} - \big(\power u{\frac{1+m}2}\big)_\rho^{(\theta)}\big|^2 }{\rho^{\frac{1+m}m}} \dx 
+
c\epsilon_o \biint_{Q_\rho^{(\theta)}} |D\power um|^2 \dxdt\\\nonumber
&\qquad\qquad+
c\bigg(\frac{ 1}{\epsilon_o^{\frac{2}{n}}}+\frac{ 1}{\epsilon_1^{\frac{2}{n}}}\bigg) \bigg[\biint_{Q_\rho^{(\theta)}} |D\power um|^{2q_o}\dxdt\bigg]^{\frac1{q_o}}
+
\frac{ c}{\epsilon_o^{\frac{2}{n}}}\biint_{Q_\rho^{(\theta)}} \big(G+|u|^{1+m}\big)\dxdt
\end{align}
for some $c=c(n,m,L,\nu,B)>0$ and any $\eps_o,\eps_1\in(0,1)$. Here we
used the fact $\frac{1-m}{1+m}<\frac2n$, which is a consequence of
our assumption $m>\frac{(n-2)_+}{n+2}$.
Selecting $\epsilon_o=\epsilon_1=\frac{\epsilon}{c}$, we also obtain
the asserted estimate \eqref{desired-est} when either $m=1$ or \eqref{intrinsic_2} is satisfied.
Therefore, the proof is complete.
\end{proof}

\section{A reverse H\"older inequality}

In this section, we are going to derive a reverse H\"older inequality
on proper cylinders.
More precisely, for $0<\rho\le1$ and $\theta>0$, we consider cylinders
$Q_{\rho}^{(\theta)}(z_o)$ that satisfy
$Q_{2\rho}^{(\theta)}(z_o)\subset\Omega_T$ and a sub-intrinsic
coupling of the type 
\begin{align}\label{Q-intrinsic}
\biint_{Q_{2\rho}^{(\theta)}(z_o)} \frac{|u|^{1+m}}{(2\rho)^{\frac{1+m}{m}}} \dxdt
\le
B\theta^{2m}
\end{align}
for some constant $B\ge1$.
In addition, we suppose that either
\begin{align}\label{Q-intrinsic-1}
\theta^{2m} \le B\biint_{Q_{\rho}^{(\theta)}(z_o)} \frac{|u|^{1+m}}{\rho^{\frac{1+m}{m}}} \dxdt
\end{align}
or
\begin{align}\label{Q-intrinsic-2}
\theta^{2m} \le B\biint_{Q_{\rho}^{(\theta)}(z_o)} \big[ |D\power um|^2+G+|u|^{1+m}\big] \dxdt
\end{align}
holds true, where
$$G=|F|^2+|g|^{1+m} +|D\power\psi m|^2+|\partial_t\power\psi m|^{\frac{1+m}m}.$$

In this situation, we have the following reverse H\"older inequality.
\begin{lemma}\label{lem:reverseHolder}
Let $\frac{(n-2)_+}{n+2}<m\le1$ and let $u\in K_\psi$ be a local weak
solution to the variational inequality \eqref{vari-ineq}, under
assumptions \eqref{obs-func}, \eqref{inhom-func}, and \eqref{A-struc}. 
Suppose that $Q_{\rho}^{(\theta)}(z_o)$ is a cylinder with
$0<\rho\le1$, $\theta>0$, and
$Q_{2\rho}^{(\theta)}(z_o)\subset\Omega_T$, for which
\eqref{Q-intrinsic} and either \eqref{Q-intrinsic-1} or
\eqref{Q-intrinsic-2} are satisfied, for some $B\ge1$.
Then we have the reverse H\"older inequality
\begin{align}
\nonumber
\biint_{Q_\rho^{(\theta)}(z_o)}|D\power um|^2\dxdt
&\le
c\bigg[\biint_{Q_{2\rho}^{(\theta)}(z_o)} |D\power um|^{2{q_o}} \dxdt\bigg]^\frac1{q_o}\\\nonumber
&\qquad+
c\biint_{Q_{2\rho}^{(\theta)}(z_o)}\big(G+|u|^{1+m}\big)\dxdt,
\end{align}
where $q_o$ is as in Lemma~\ref{lem:sobolev-poin}.
Here, the dependencies of the constant are given by $c=c(n,m,\nu,L,B)$.
\end{lemma}
\begin{proof}
We omit the center point $z_o$ throughout the proof. 
Let us consider radii $r,s$ with $\rho\le r<s\le2\rho$.
Then, assumptions \eqref{Q-intrinsic}-\eqref{Q-intrinsic-2} imply that
the cylinder $Q_s^{(\theta)}$ satisfies
\eqref{sub-intrinsic-1}-\eqref{intrinsic_2} with $2^{n+2+\frac2m}B$
instead of $B$. Moreover, Lemma~\ref{lem:energy} with $\power
a{\frac{1+m}2}=(\power u{\frac{1+m}2})_r^{(\theta)}$ yields the bound
\begin{align}\label{i-ii-iii}
\nonumber
&\sup_{t\in\Lambda_r}\bint_{B_r^{(\theta)}}\frac{\big|\power u{\frac{1+m}2}(\cdot,t)-(\power u{\frac{1+m}2})_r^{(\theta)}\big|^2}{r^\frac{1+m}m}\dx
+
\biint_{Q_r^{(\theta)}}|D\power um|^2\dxdt\\\nonumber
&\le
c\,\biint_{Q_s^{(\theta)}}\frac{\big|\power u{\frac{1+m}2}-(\power u{\frac{1+m}2})_r^{(\theta)}\big|^2}{s^\frac{1+m}m-r^\frac{1+m}m}\dxdt\\\nonumber
&\qquad+
c\theta^{\frac{2m(1-m)}{1+m}}\biint_{Q_s^{(\theta)}}\frac{\big|\power um-
\power{\big[( \power u{\frac{1+m}2})_r^{(\theta)}\big]}{\frac{2m}{1+m}} \big|^2}{(s-r)^2}\dxdt\\
&\qquad+
c\,\biint_{Q_s^{(\theta)}}\big(G+|u|^{1+m}\big)\dxdt.
\end{align}
For the first two terms on the right-hand side, we introduce the
abbreviations 
\begin{align*}
{\rm{I}}
:=\biint_{Q_s^{(\theta)}}\frac{\big|\power u{\frac{1+m}2}-(\power u{\frac{1+m}2})_r^{(\theta)}\big|^2}{s^\frac{1+m}m-r^\frac{1+m}m}\dxdt,
\end{align*}
\begin{align*}
{\rm{II}}
:=\theta^{\frac{2m(1-m)}{1+m}}\biint_{Q_s^{(\theta)}}\frac{\big|\power um-
\power{\big[( \power u{\frac{1+m}2})_r^{(\theta)}\big]}{\frac{2m}{1+m}} \big|^2}{(s-r)^2}\dxdt,
\end{align*}
and let
\begin{align*}
M_{r,s}:=\frac{s}{s-r}.
\end{align*}
By Lemma~\ref{lem:mvint} and~\eqref{a-b-2},
the term $\rm{I}$ is bounded by
\begin{align}\label{est-i}
{\rm{I}}
\le
cM_{r,s}^\frac{1+m}m\biint_{Q_s^{(\theta)}}\frac{\big|\power u{\frac{1+m}2}-(\power u{\frac{1+m}2})_s^{(\theta)}\big|^2}{s^\frac{1+m}m}\dxdt.
\end{align}
We now consider the term $\rm{II}$. First, when \eqref{Q-intrinsic-2}
is satisfied, we use \eqref{a-b-2}, H\"older's inequality, and
Lemma~\ref{lem:mvint}, with the result
\begin{align}\label{est-ii}
{\rm{II}}
&\le
cM_{r,s}^2\theta^{\frac{2m(1-m)}{1+m}}\bigg[\biint_{Q_s^{(\theta)}}\frac{\big|\power u{\frac{1+m}2}-( \power u{\frac{1+m}2})_s^{(\theta)} \big|^2}{s^\frac{1+m}m}\dxdt\bigg]^\frac{2m}{1+m}.
\end{align}
Next, in the case $m<1$ we apply Young's inequality with
exponents $\frac{1+m}{1-m}$ and $\frac{1+m}{2m}$, and then \eqref{Q-intrinsic-2} to obtain
\begin{align}\label{est-ii-1}
  {\rm{II}}
  & 
    \le \frac{\eps(1-m)}{B}M_{r,s}^2\theta^{2m}
  +
  \frac{cBM_{r,s}^2}{\epsilon^\frac{1-m}{2m}}\biint_{Q_s^{(\theta)}}\frac{\big|\power
    u{\frac{1+m}2}-( \power u{\frac{1+m}2})_s^{(\theta)}
    \big|^2}{s^\frac{1+m}m}\dxdt
\\
&\le
\epsilon M_{r,s}^2\biint_{Q_s^{(\theta)}}\big(|D\power um|^2+|u|^{1+m}+G\big)\dxdt\nonumber\\
&\qquad+
\frac{cM_{r,s}^2}{\epsilon^\frac{1-m}{2m}}\biint_{Q_s^{(\theta)}}\frac{\big|\power u{\frac{1+m}2}-( \power u{\frac{1+m}2})_s^{(\theta)} \big|^2}{s^\frac{1+m}m}\dxdt
\end{align}
for any $0<\epsilon<1$.
Note that in the case $m=1$, the same estimate immediately follows from
\eqref{est-ii} without the
application of Young's inequality and \eqref{Q-intrinsic-2}. 
Next, when \eqref{Q-intrinsic-1} is satisfied, we deduce
\begin{align*}
\theta^{2m}
&\le
2^{n+2+\frac2m}B\biint_{Q_s^{(\theta)}} \frac{|u|^{1+m}}{s^\frac{1+m}m}\dxdt\\
&\le
c\,\biint_{Q_s^{(\theta)}}\frac{\big|\power u{\frac{1+m}2}-( \power u{\frac{1+m}2})_r^{(\theta)} \big|^2}{s^\frac{1+m}m}\dxdt
+
\frac{c\big|( \power u{\frac{1+m}2})_r^{(\theta)}\big|^2}{s^\frac{1+m}m}.
\end{align*}
This implies that
\begin{align}
{\rm{II}}
&\le
cM_{r,s}^2 (\rm{J_1+J_2}),
\end{align}
where
$$
{\rm{J_1}}
:=
\bigg[ \biint_{Q_s^{(\theta)}}\frac{\big|\power u{\frac{1+m}2}-( \power u{\frac{1+m}2})_r^{(\theta)} \big|^2}{s^\frac{1+m}m}\dxdt \bigg]^{\frac{1-m}{1+m}}
\biint_{Q_s^{(\theta)}}\frac{\big|\power um-
\power{[(  u^{\frac{1+m}2})_r^{(\theta)}]}{\frac{2m}{1+m}} \big|^2}{s^2}\dxdt
$$
and
$$
{\rm{J_2}}
:=
\bigg[ \frac{\big|( \power u{\frac{1+m}2})_r^{(\theta)}\big|^2}{s^\frac{1+m}m}\bigg]^{\frac{1-m}{1+m}}
\biint_{Q_s^{(\theta)}}\frac{\big|\power um-
\power{[(u^{\frac{1+m}2})_r^{(\theta)}]}{\frac{2m}{1+m}} \big|^2}{s^2}\dxdt.
$$
We then use \eqref{a-b-2}, H\"older's inequality, and Lemma~\ref{lem:mvint} for the estimate of $J_1$, while to the term $J_2$ we apply \eqref{a-b-1} and Lemma~\ref{lem:mvint}. In this way we find
\begin{align}
{\rm{J_1+J_2}}
\le
c\, \biint_{Q_s^{(\theta)}}\frac{\big|\power u{\frac{1+m}2}-( \power u{\frac{1+m}2})_s^{(\theta)} \big|^2}{s^\frac{1+m}m}\dxdt ,
\end{align}
and hence
\begin{align}\label{est-ii-2}
{\rm{II}}
\le
cM_{r,s}^2 \biint_{Q_s^{(\theta)}}\frac{\big|\power u{\frac{1+m}2}-( \power u{\frac{1+m}2})_s^{(\theta)} \big|^2}{s^\frac{1+m}m}\dxdt.
\end{align}
In view of \eqref{est-ii-1} and \eqref{est-ii-2},
we obtain in any case 
\begin{align}\label{est-ii-f}
{\rm{II}}
&\le
\epsilon M_{r,s}^2\biint_{Q_s^{(\theta)}}\big(|D\power um|^2+|u|^{1+m}+G\big)\dxdt\nonumber\\
&\qquad+
\frac{cM_{r,s}^2}{\epsilon^\frac{1-m}{2m}}\biint_{Q_s^{(\theta)}}\frac{\big|\power u{\frac{1+m}2}-( \power u{\frac{1+m}2})_s^{(\theta)} \big|^2}{s^\frac{1+m}m}\dxdt,
\end{align}
provided either \eqref{Q-intrinsic-1} or \eqref{Q-intrinsic-2}.
Inserting \eqref{est-i} and \eqref{est-ii-f} into \eqref{i-ii-iii}, we therefore have
\begin{align}
\nonumber
&\sup_{t\in\Lambda_r}\bint_{B_r^{(\theta)}}\frac{\big|\power u{\frac{1+m}2}(\cdot,t)-(\power u{\frac{1+m}2})_r^{(\theta)}\big|^2}{r^\frac{1+m}m}\dx
+
\biint_{Q_r^{(\theta)}}|D\power um|^2\dxdt\\\nonumber
&\qquad\le
\epsilon M_{r,s}^2\biint_{Q_s^{(\theta)}} |D\power um|^2 \dxdt
+
\frac{cM_{r,s}^\frac{1+m}m}{\epsilon^\frac{1-m}{2m}}\biint_{Q_s^{(\theta)}}\frac{\big|\power u{\frac{1+m}2}-( \power u{\frac{1+m}2})_s^{(\theta)} \big|^2}{s^\frac{1+m}m}\dxdt\\\nonumber
&\qquad\qquad+
cM_{r,s}^2\biint_{Q_s^{(\theta)}}\big(G+|u|^{1+m}\big)\dxdt.
\end{align}
Now, we employ Lemma~\ref{lem:sobolev-poin} with $\epsilon$ replaced
by $\epsilon^\frac{1+m}{2m}$ for the second integral on the right-hand side. This leads to
\begin{align}
\nonumber
&\sup_{t\in\Lambda_r}\bint_{B_r^{(\theta)}}\frac{\big|\power u{\frac{1+m}2}(\cdot,t)-(\power u{\frac{1+m}2})_r^{(\theta)}\big|^2}{r^\frac{1+m}m}\dx
+
\biint_{Q_r^{(\theta)}}|D\power um|^2\dxdt\\\nonumber
&\qquad\le
\epsilon c M_{r,s}^\frac{1+m}m\bigg[\sup_{t\in\Lambda_s}\bint_{B_s^{(\theta)}}\frac{\big|\power u{\frac{1+m}2}(\cdot,t)-(\power u{\frac{1+m}2})_s^{(\theta)}\big|^2}{s^\frac{1+m}m}\dx
+
\biint_{Q_s^{(\theta)}}|D\power um|^2\dxdt\bigg]\\\nonumber
&\qquad\qquad+
\frac{c M_{r,s}^\frac{1+m}m}{\epsilon^\frac{2+2m+n(1-m)}{2mn}}\Bigg\{\bigg[\biint_{Q_s^{(\theta)}} |D\power um|^{2q_o} \dxdt\bigg]^\frac1{q_o}\\\nonumber
&\qquad\qquad+
\biint_{Q_s^{(\theta)}}\big(G+|u|^{1+m}\big)\dxdt\Bigg\}.
\end{align}
Choosing $\frac1\epsilon=2cM_{r,s}^\frac{1+m}m$, we arrive at 
\begin{align}
\nonumber
&\sup_{t\in\Lambda_r}\bint_{B_r^{(\theta)}}\frac{\big|\power u{\frac{1+m}2}(\cdot,t)-(\power u{\frac{1+m}2})_r^{(\theta)}\big|^2}{r^\frac{1+m}m}\dx
+
\biint_{Q_r^{(\theta)}}|D\power um|^2\dxdt\\\nonumber
&\le
\frac12\bigg[\sup_{t\in\Lambda_s}\bint_{B_s^{(\theta)}}\frac{\big|\power u{\frac{1+m}2}(\cdot,t)-(\power u{\frac{1+m}2})_s^{(\theta)}\big|^2}{s^\frac{1+m}m}\dx
+
\biint_{Q_s^{(\theta)}}|D\power um|^2\dxdt\bigg]\\\nonumber
&\qquad+
c\bigg(\frac{\rho}{s-r}\bigg)^\frac{(n+2)(1+m)^2}{2nm^2}\Bigg\{\bigg[\biint_{Q_{2\rho}^{(\theta)}} |D\power um|^{2q_o} \dxdt\bigg]^\frac1{q_o}\\\nonumber
&\qquad+
\biint_{Q_{2\rho}^{(\theta)}}\big(G+|u|^{1+m}\big)\dxdt\Bigg\},
\end{align}
and this estimate holds for any $r,s$ with $\rho\le r<s\le2\rho$.
Therefore, we can use Lemma~\ref{lem:tech} to absorb the first term on
the right-hand side. This completes the proof .
\end{proof}

We next provide an auxiliary estimate in the case of an intrinsic
cylinder with the properties \eqref{Q-intrinsic} and \eqref{Q-intrinsic-1} with $B=1$.

\begin{lemma}\label{lem:theta_m}
Let $\frac{(n-2)_+}{n+2}<m\le1$ and $u\in K_\psi$ be a local weak
solution to the variational inequality \eqref{vari-ineq}, where
assumptions \eqref{obs-func}, \eqref{inhom-func}, and \eqref{A-struc} are in force. 
Suppose that $Q_{\rho}^{(\theta)}(z_o)$ is a cylinder with
$0<\rho\le1$, $\theta>0$, and
$Q_{2\rho}^{(\theta)}(z_o)\subset\Omega_T$, which fulfills \eqref{Q-intrinsic} and \eqref{Q-intrinsic-1} with $B=1$. Then we have
\begin{align}
\nonumber
\theta^m
\le\frac1{\sqrt2}\bigg[\biint_{Q_{\rho/2}^{(\theta)}(z_o)} \frac{|u|^{1+m}}{(\rho/2)^\frac{1+m}m}\dxdt\bigg]^\frac12 
+
c\bigg[ \biint_{Q_{2\rho}^{(\theta)}(z_o)} \big( |D\power um|^2 +G + |u|^{1+m} \big) \dxdt\bigg]^\frac12
\end{align}
for some $c=c(n,m,\nu,L)>0$.
\end{lemma}
\begin{proof}
  The asserted estimate follows from the energy estimate in
  Lemma~\ref{lem:energy} and the Sobolev-Poincar\'e estimate
  established in
  Lemma~\ref{lem:sobolev-poin} by the exact same arguments as in
  \cite[Lemma 6.2]{BDS-singular-higher-int}. 
  We just need to replace $|F|^2$ by $G+|u|^{1+m}$ and to apply
  Lemma~\ref{lem:energy} and Lemma~\ref{lem:sobolev-poin} instead of
  \cite[Lemma 4.1]{BDS-singular-higher-int} and \cite[Lemma
    5.1]{BDS-singular-higher-int}. 
\end{proof}

\section{Proof of the higher integrability}\label{sec:hi}
In this section we prove Theorem~\ref{thm:main} by establishing higher
integrability of the spatial gradient for solutions of our obstacle problems. 
In order to specify the setting, we fix a cylinder
$Q_{8R}(x_o,t_o)\equiv B_{8R}(x_o)\times \big(t_o-(8R)^\frac{1+m}{m},
t_o+(8R)^\frac{1+m}{m}\big)\subset\Omega_T$ with $0<R\le1$.
In the following, we omit the center in the notation and
write $Q_{\rho}:=Q_{\rho}(x_o,t_o)$ for short, with
$\rho\in(0,8R]$. Next, we define
\begin{align}\label{def:lambda_o}
	\lambda_o
	:=
	\Bigg[\biint_{Q_{4R}} \bigg(
	\frac{\abs{u}^{1+m}}{(4 R)^\frac{1+m}m} +
	\big|D\power{u}{m}\big|^2 +
	G+1\bigg) \dxdt\Bigg]^{\frac{d}{2m}},
\end{align}
where $G$ and $d$ are as in \eqref{def-G} and \eqref{def-d}, respectively.
Recalling definition \eqref{def-Q} of the cylinders $Q_\rho^{(\theta)}(z)$, we notice that $Q_\rho^{(\theta)}(z)\subset Q_{4R}$ for all $z\in Q_{2R}$, $\rho\in(0,R]$, and $\theta\ge 1$.

Now we introduce a system of sub-intrinsic cylinders that reflect the
anisotropic scaling behaviour of the singular porous medium equation
and are suitable for the derivation of gradient estimates.
The idea of the construction goes back to
\cite{ Schwarzacher, Gianazza-Schwarzacher, Gianazza-Schwarzacher:m<1}. 
Here, we follow the presentation from 
\cite{BDS-singular-higher-int} to briefly recall the definition of the
cylinders and their properties.
For a point $z_o\in Q_{2R}$ and a radius $\rho\in (0,R]$ we define
\begin{equation}\label{def-theta-tilde}
  \widetilde\theta_{z_o;\rho}
  :=
  \inf\bigg\{\theta\in[\lambda_o,\infty):
  \biint_{Q^{(\theta)}_\rho(z_o)}
  \frac{\abs{u}^{1+m}}{\rho^{\frac{1+m}m}}\dx\dt 
  \le 
  \theta^{2m} \bigg\}.
\end{equation}
We point out that for this definition,
the restriction $m>\frac{(n-2)_+}{n+2}$ is crucial to ensure that the
infimum is not taken over the empty set. This becomes evident after
re-writing the condition in the above infimum to
\begin{equation*}
   \frac{1}{|Q_\rho|}
   \iint_{Q^{(\theta)}_\rho(z_o)} 
   \frac{\abs{u}^{1+m}}{\rho^{\frac{1+m}m}} \dx\dt 
   \le 
   \theta^{\frac{2m}{d}},
\end{equation*} 
with the scaling deficit $d$ defined in ~\eqref{def-d}. For
exponents $m>\frac{(n-2)_+}{n+2}$, this scaling deficit satisfies
$d>0$. This implies that the right-hand side blows up
when $\theta\to\infty$, while the left-hand side tends to
zero. Therefore, the parameters $\widetilde\theta_{z_o;\rho}$ are
well-defined for our range of exponents.
Next, since the mapping $(0,R]\ni\rho\to \widetilde\theta_\rho$ might not be
decreasing, we replace it by a decreasing variant defined via
$$
	\theta_{z_o;\rho}
	:=
	\max_{r\in[\rho,R]} \widetilde\theta_{z_o;r}
        \qquad\mbox{for }\rho\in(0,R].
$$
Moreover, we define associated radii by
\begin{equation}\label{rho-tilde}
	\widetilde\rho
	:=
	\left\{
	\begin{array}{cl}
	R, &
	\quad\mbox{if $\theta_\rho=\lambda_o$,} \\[5pt]
	\min\big\{s\in[\rho, R]: \theta_s=\widetilde \theta_s \big\}, &
	\quad\mbox{if $\theta_\rho>\lambda_o$.}
	\end{array}
	\right.
\end{equation}
For the derivation of the gradient estimate we work with the cylinders
$Q_{\rho}^{(\theta_{z_o;\rho})}(z_o)$.
Due to the monotonic dependence of the
functions $\theta_{z_o;\rho}$ on $\rho$, these cylinders are nested in
the sense that
\begin{equation*}
  Q_{\rho}^{(\theta_{z_o;\rho})}(z_o)\subset
  Q_{r}^{(\theta_{z_o;r})}(z_o)
  \qquad\mbox{whenever }0<\rho<r\le R.
\end{equation*}
A detailed analysis of the properties of these cylinders and the
functions $\theta_{z_o;\rho}$ can be found
in \cite[Section 7.1]{BDS-singular-higher-int}. Here, we only state
the properties that are necessary for our purposes and refer the reader
to \cite[Section 7.1]{BDS-singular-higher-int} for the proofs.

The mapping $\rho\mapsto\theta_{z_o;\rho}\in[\lambda_o,\infty)$ is
continuous and monotonically decreasing and the corresponding
cylinders $Q_s^{(\theta_{z_o;\rho})}(z_o)$ are sub-intrinsic for all
$s\in[\rho,R]$ in the sense 
\begin{align}\label{sub-intrinsic-rho}
	\biint_{Q_{s}^{(\theta_{z_o;\rho})}(z_o)} \frac{\abs{u}^{1+m}}{s^\frac{1+m}m} \dxdt
	\le 
	\theta_{z_o;\rho}^{2m}
	\quad\mbox{for all $s\in[\rho, R]$.}
\end{align}
For the radius $\Tilde{\rho}\in[\rho,R]$ defined in \eqref{rho-tilde}
it holds 
\begin{align}\label{tilde-rho}
\theta_{z_o;s}=\theta_{z_o;\rho}\quad\mbox{for all $s\in[\rho,\Tilde{\rho}]$}.
\end{align}
Moreover, we have either that
\begin{align}\label{intrinsic-tilde-rho}
	\biint_{Q_{\Tilde{\rho}}^{(\theta_{z_o;\rho})}(z_o)} \frac{\abs{u}^{1+m}}{\Tilde{\rho}^\frac{1+m}m} \dxdt
	= 
	\theta_{z_o;\rho}^{2m}
\quad
\mbox{and}
\quad
\theta_{z_o;\rho}>\lambda_o
\end{align}
or that 
\begin{align}\label{tilde-rho-R}
\Tilde{\rho}=R\quad\mbox{and}\quad \theta_{z_o;\rho}=\lambda_o.
\end{align}
Note that in any case, the cylinder $Q_{\Tilde{\rho}}^{(\theta_{z_o;\rho})}(z_o)$
is sub-intrinsic according to \eqref{sub-intrinsic-rho} with $s=\Tilde\rho$.
Furthermore, we have the estimates 
\begin{align}\label{theta-compare}
\theta_{z_o;s}
\le
\theta_{z_o;\rho}
\le
\bigg(\frac{s}{\rho}\bigg)^{\frac{d}{2m}(n+2+\frac2m)}\theta_{z_o;s}, \quad\mbox{whenever $0<\rho\le s\le R$}
\end{align}
and 
\begin{align}\label{theta-R}
\theta_{z_o;R}\le4^{\frac{d}{2m}(n+2+\frac2m)}\lambda_o.
\end{align}
The proofs of \eqref{sub-intrinsic-rho}-\eqref{theta-R} can be found in \cite[Section 7.1]{BDS-singular-higher-int}.
Finally, we recall the following Vitali type covering property of the
above cylinders, which has first been observed in
\cite{Gianazza-Schwarzacher}. We state a version that has
been established in \cite[Lemma 7.1]{BDS-singular-higher-int}.
\begin{lemma}\label{lem:vitali}
Let $\Omega$ be a bounded domain in $\R^n$ and let
$Q_{\rho}^{(\theta_{z;\rho})}(z)\subset\Omega$ be a system of
cylinders with the properties \eqref{sub-intrinsic-rho}-\eqref{theta-R}.
Then there exists a constant $ c_o= c_o(n,m)\ge 20$ 
such that for any collection $\mathcal X$ of cylinders $Q_{4\rho}^{(\theta_{z;\rho})}(z)$ with radius $\rho\in(0,\tfrac{R}{ c_o})$ there is a countable subcollection $\mathcal Y$ of disjoint cylinders in $\mathcal X$ such that  
\begin{equation}\label{covering}
	\bigcup_{Q\in\mathcal X} Q
	\subset 
	\bigcup_{Q\in\mathcal Y} \Tilde Q,
\end{equation}
where $\Tilde Q$ is the $\frac{c_o}{4}$-times enlarged cylinder $Q$, i.e.~$\Tilde Q=Q_{c_o \rho}^{(\theta_{z;\rho})}(z)$ if $Q=Q_{4\rho}^{(\theta_{z;\rho})}(z)$.
\end{lemma}

Now we are in a position to prove our main theorem.

\begin{proof}
Let us define upper level sets of $|D\power{u}{m}|$ by 
$$
	\boldsymbol E(\rho,\lambda)
	:=
	\Big\{z\in Q_{\rho}: 
	\mbox{$z$ is a Lebesgue point of $\big|D\power{u}{m}\big|$ and 
	$\big|D\power{u}{m}\big|(z) > \lambda^{m}$}\Big\}
$$
for $\rho\in(0,2R]$ and $\lambda>\lambda_o$.
Here, the notion of Lebesgue point is to be understood with respect to
the family of parabolic cylinders
$Q_{\rho}^{(\theta_{z_o;\rho})}(z_o)$ introduced above.
By virtue of the Vitali type covering lemma~\ref{lem:vitali}, 
Lebesgue's differentiation theorem also holds for this family of
parabolic cylinders, cf. \cite[2.9.1]{Federer}. Therefore, by the restriction to
Lebesgue points in the definition of $\boldsymbol E(\rho,\lambda)$
we only exclude a set of vanishing Lebesgue measure.

Moreover, we fix radii $R\le r_1<r_2\le 2R$.
We note that 
$
	Q_\rho^{(\theta)}(z_o) \subset Q_\rho (z_o) \subset Q_{r_2}\subset Q_{2R}
$
for any $z_o\in Q_{r_1}$, $\theta\ge1$ and $\rho\in(0,r_2-r_1]$.
Suppose that $z_o\in E(r_1,\lambda)$. Then we have from the definition of $E(r_1,\lambda)$
\begin{equation}\label{larger-lambda}
	\lim_{s\downarrow 0} 
	\biint_{Q_{s}^{(\theta_{z_o;s})}(z_o)} 
	\Big(\big|D\power{u}{m}\big|^2 + G+|u|^{1+m} \Big) 
	\dxdt
	\ge
	\big|D\power{u}{m}\big|^2(z_o)
	>
	\lambda^{2m}.
\end{equation}
Moreover, for any $s\in[\frac{r_2-r_1}{c_o},R]$, where $c_o\ge20$
denotes the constant from Lemma~\ref{lem:vitali}, we deduce by the
definition of $\lambda_o$, \eqref{theta-compare}, and \eqref{theta-R}
\begin{align}\label{less-lambda}
\nonumber
	\biint_{Q_{s}^{(\theta_{z_o;s})}(z_o)} 
	\Big(\big|D\power{u}{m}\big|^2 + G+|u|^{1+m}\Big)\dxdt
	&\le
	4^\frac{1+m}m\frac{|Q_{4R}|}{|Q_{s}|}\theta_{z_o;s}^{\frac{nm(1-m)}{1+m}}
	\lambda_o^\frac{2m}d \\\nonumber
	&\le
	4^\frac{1+m}m\Big(\frac{4R}{s}\Big)^\frac{d(n+2)(1+m)}{2m}
	\lambda_o^{2m} \\
	&\le
	N^{2m} \lambda_o^{2m} 
	<
	\lambda^{2m},
\end{align}
provided $\lambda>N\lambda_o$,
where
\begin{equation}\label{choice_lambda}
	N
	:=
	4^\frac{1+m}m\Big(\frac{4 c_o R}{r_2-r_1}\Big)^{\frac{d(n+2)(1+m)}{4m^2}}
	>1.
\end{equation}
For the following argument, we only consider parameters $\lambda>N\lambda_o$.
We now combine \eqref{larger-lambda} and \eqref{less-lambda}, and we
also observe that the left-hand side of \eqref{less-lambda}
continuously depends on $s$ by the continuity of the mapping
$s\mapsto\theta_{z_o;s}$ and the absolute continuity of the
integral. Therefore, for any $z_o\in E(r_1,\lambda)$
the intermediate value theorem guarantees the
existence of a radius $\rho_{z_o} \in(0, \frac{r_2-r_1}{c_o})$ such that
\begin{align}\label{=lambda}
	\biint_{Q_{\rho_{z_o}}^{(\theta_{z_o;\rho_{z_o}})}(z_o)} 
	\Big(\big|D\power{u}{m}\big|^2 + G+|u|^{1+m}\Big) \dxdt
	=
	\lambda^{2m}
\end{align}
and
\begin{align}\label{<lambda}
	\biint_{Q_{s}^{(\theta_{z_o;s})}(z_o)} 
	\Big(\big|D\power{u}{m}\big|^2 + G+|u|^{1+m}\Big) \dxdt
	<
	\lambda^{2m}
	\qquad
	\mbox{for all $s\in (\rho_{z_o}, R]$.}
\end{align}
Thanks to \eqref{theta-compare}, \eqref{=lambda}, and \eqref{<lambda}, we obtain for any $\rho_{z_o}\le s<\tau\le R$ that
\begin{align}\label{<lambda-1}
	&\biint_{Q_{\tau}^{(\theta_{z_o;s})}(z_o)} 
	\Big(\big|D\power{u}{m}\big|^2 + G+|u|^{1+m}\Big) \dxdt\\\nonumber
	&\le
	\bigg(\frac{\theta_{z_o;s}}{\theta_{z_o;\tau}}\bigg)^\frac{nm(1-m)}{1+m}
	\biint_{Q_{\tau}^{(\theta_{z_o;\tau})}(z_o)} 
	\Big(\big|D\power{u}{m}\big|^2 + G+|u|^{1+m}\Big) \dxdt\\\nonumber
	&\le
	\bigg(\frac\tau s \bigg)^{\frac{nd(1-m)}{2(1+m)}(n+2+\frac2m)} \lambda^{2m}.
\end{align}

\subsection{A Reverse H\"older Inequality}
For a fixed point $z_o\in E(r_1,\lambda)$ and $\lambda>N\lambda_o$, we consider the
radius
$\rho_{z_o}\in(0,\frac{r_2-r_1}{c_o})$ constructed above, which
satisfies in particular \eqref{=lambda} and \eqref{<lambda}.
Our next goal is to establish a reverse H\"older inequality on the
cylinder $Q_{2\rho_{z_o}}^{(\theta_{z_o;\rho_{z_o}})}(z_o)$ by using
Lemma~\ref{lem:reverseHolder}. To do this, we need to verify that the
assumptions of Lemma~\ref{lem:reverseHolder} are fulfilled on the
cylinder $Q_{2\rho_{z_o}}^{(\theta_{z_o;\rho_{z_o}})}(z_o)$.
We first observe that \eqref{sub-intrinsic-rho} with $s=4\rho_{z_o}$
implies 
\begin{align}
\biint_{Q_{4\rho_{z_o}}^{(\theta_{z_o;\rho_{z_o}})}(z_o)} \frac{|u|^{1+m}}{(4\rho_{z_o})^\frac{1+m}m}\dxdt
\le
\theta_{z_o;\rho_{z_o}}^{2m},
\end{align}
which implies that assumption \eqref{Q-intrinsic} holds with the constant $B=1$.
We next consider the radius $\Tilde \rho_{z_o}\in[\rho_{z_o},R]$,
which has the properties \eqref{tilde-rho}-\eqref{tilde-rho-R},
and distinguish between the cases $\Tilde\rho_{z_o}\le2\rho_{z_o}$ and $\Tilde\rho_{z_o}>2\rho_{z_o}$.
In the first case, we deduce from \eqref{intrinsic-tilde-rho} that
\begin{align}
\theta_{z_o;\rho_{z_o}}^{2m}
=
\biint_{Q_{\Tilde\rho_{z_o}}^{(\theta_{z_o;\rho_{z_o}})}(z_o)} \frac{|u|^{1+m}}{\Tilde\rho_{z_o}^\frac{1+m}m}\dxdt
\le
2^{n+2+\frac2m}
\biint_{Q_{{2\rho}_{z_o}}^{(\theta_{z_o;\rho_{z_o}})}(z_o)} \frac{|u|^{1+m}}{(2\rho_{z_o})^\frac{1+m}m}\dxdt
\end{align}
holds true. 
This implies \eqref{Q-intrinsic-1} with
$B=2^{n+2+\frac2m}$. Therefore, Lemma~\ref{lem:reverseHolder} is
applicable in the case $\Tilde\rho_{z_o}\le 2\rho_{z_o}$.
It remains to consider the case $\Tilde\rho_{z_o}>2\rho_{z_o}$.
In this case, we claim that 
\begin{align}\label{claim}
\theta_{z_o;\rho_{z_o}}\le c\lambda
\end{align}
holds true for some $c=c(n,m,\nu,L)>0.$
If $\Tilde\rho_{z_o}>\frac R2$, we use \eqref{theta-compare} and
\eqref{theta-R} to deduce
$
\theta_{z_o;\rho_{z_o}}
=
\theta_{z_o;\Tilde\rho_{z_o}}\le8^{\frac{d}{2m}(n+2+\frac2m)}\lambda_o†
\le
8^{\frac{d}{2m}(n+2+\frac2m)}\lambda,
$
which implies the claim \eqref{claim}.
In the case $\Tilde\rho_{z_o}\le\frac R2$, however, we observe 
that $Q_{\Tilde\rho_{z_o}}^{(\theta_{z_o;\rho_{z_o}})}(z_o)$ is
intrinsic by \eqref{intrinsic-tilde-rho} and
$Q_{2\Tilde\rho_{z_o}}^{(\theta_{z_o;\rho_{z_o}})}(z_o)$ is
sub-intrinsic by \eqref{sub-intrinsic-rho} with $s=2\Tilde\rho_{z_o}\in[\rho_{z_o},R]$.  
Hence, the cylinder
$Q_{\Tilde\rho_{z_o}}^{(\theta_{z_o;\rho_{z_o}})}(z_o)$ satisfies the
assumptions of Lemma~\ref{lem:theta_m}, which implies
\begin{align}
\nonumber
\theta_{z_o;\rho_{z_o}}^m
&\le
\frac1{\sqrt2}\bigg[\biint_{Q_{\Tilde\rho_{z_o}/2}^{(\theta_{z_o;\rho_{z_o}})}(z_o)} \frac{|u|^{1+m}}{(\Tilde\rho_{z_o}/2)^\frac{1+m}m}\dxdt\bigg]^{\frac12} \\
&\qquad+
c\bigg[ \biint_{Q_{2\Tilde\rho_{z_o}}^{(\theta_{z_o;\rho_{z_o}})}(z_o)} \big(|D\power um|^2 +G + |u|^{1+m} \big)\dxdt\bigg]^{\frac12}.
\end{align}
Since $\Tilde\rho_{z_o}/2\in[\rho_{z_o},R]$ we can estimate the first term
on the right-hand side by means of \eqref{sub-intrinsic-rho}. Moreover,
observing $\rho_{z_o}\le\Tilde\rho_{z_o}<2\Tilde\rho_{z_o}\le R$ and
$\theta_{z_o;\rho_{z_o}}=\theta_{z_o;\Tilde\rho_{z_o}}$, we estimate the last term on the right-hand side by
\eqref{<lambda-1} with $s=\Tilde\rho_{z_o}$ and
$\tau=2\Tilde\rho_{z_o}$. In this
way we deduce
\begin{align}
\theta_{z_o;\rho_{z_o}}^m
&\le
\frac1{\sqrt2}\theta_{z_o;\rho_{z_o}}^m
+
c\lambda^m,
\end{align}
which yields \eqref{claim} by re-absorbing $\frac1{\sqrt2}\theta_{z_o;\rho_{z_o}}^m$ into the left-hand side.
Therefore we obtain the claim \eqref{claim} in any case.

Now we combine \eqref{=lambda} and \eqref{claim} to get
\begin{align}
\theta_{z_o;\rho_{z_o}}^{2m}
\le c\lambda^{2m}
\le
c\,\biint_{Q_{2\rho_{z_o}}^{(\theta_{z_o;\rho_{z_o}})}(z_o)}\big( |D\power um|^2 +G + |u|^{1+m} \big)\dxdt.
\end{align}
This implies that in the case of $\Tilde\rho_{z_o}>2\rho_{z_o}$, the cylinder $Q_{2\rho_{z_o}}^{(\theta_{z_o;\rho_{z_o}})}(z_o)$ fulfills assumption \eqref{Q-intrinsic-2}.
Consequently, we have verified the hypotheses of Lemma~\ref{lem:reverseHolder} in any case. Thus, this lemma yields the reverse H\"older inequality
\begin{align}\label{reverseHolder}
\nonumber
\biint_{Q_{2\rho_{z_o}}^{(\theta_{z_o;\rho_{z_o}})}(z_o)}|D\power um|^2\dxdt
&\le
c\bigg[\biint_{Q_{4\rho_{z_o}}^{(\theta_{z_o;\rho_{z_o}})}(z_o)} |D\power um|^{2q_o} \dxdt\bigg]^\frac1{q_o}\\
&\qquad+
c\,\biint_{Q_{4\rho_{z_o}}^{(\theta_{z_o;\rho_{z_o}})}(z_o)}\big(G+|u|^{1+m}\big)\dxdt,
\end{align}
where 
\begin{align*}
q_o:=\bigg\{\frac12,\frac{n(1+m)}{2(nm+1+m)}\bigg\}<1.
\end{align*}

\subsection{Estimate on upper level sets}
So far, we have verified that for each $z_o\in \boldsymbol E(r_1,\lambda)$ with $\lambda>N\lambda_o$ there is a cylinder $Q_{\rho_{z_o}}^{(\theta_{z_o;\rho_{z_o}})}(z_o)$ such that \eqref{=lambda}-\eqref{<lambda-1} and \eqref{reverseHolder} are satisfied.
Now we define the upper level set  
$$
	\boldsymbol F(\rho,\lambda)
	:=
	\Big\{z\in Q_{\rho}: 
	\mbox{$z$ is a Lebesgue point of $(G+|u|^{1+m})$ and 
	$G+|u|^{1+m}>\lambda^{2m}$}\Big\}.
$$
According to \eqref{=lambda} and \eqref{reverseHolder} we have for any $\eta\in (0,1]$  
\begin{align}\label{est-lambda2m}
\nonumber
	\lambda^{2m}
	&=
	\biint_{Q_{\rho_{z_o}}^{(\theta_{z_o;\rho_{z_o}})}(z_o)} 
	\big(\big|D\power{u}{m}\big|^2 + G+|u|^{1+m}\big) \dxdt \\\nonumber
	&\le
	c\,\bigg[\biint_{Q_{4\rho_{z_o}}^{(\theta_{z_o;\rho_{z_o}})}(z_o)} 
	\big|D\power{u}{m}\big|^{2q_o} \dxdt \bigg]^{\frac{1}{q_o}} +
	c\,  \biint_{Q_{4\rho_{z_o}}^{(\theta_{z_o;\rho_{z_o}})}(z_o)} (G+|u|^{1+m}) \dxdt\\\nonumber
	&\le
	c\,\eta^{2m}\lambda^{2m} +
	c\,\Bigg[\frac{1}{\big|Q_{4\rho_{z_o}}^{(\theta_{z_o;\rho_{z_o}})}(z_o)\big|}
	\iint_{Q_{4\rho_{z_o}}^{(\theta_{z_o;\rho_{z_o}})}(z_o)\cap \boldsymbol E(r_2,\eta\lambda)} 
	\big|D\power{u}{m}\big|^{2q_o} \dxdt \Bigg]^{\frac{1}{q_o}} \\
	&\quad+
	\frac{c}{\big|Q_{4\rho_{z_o}}^{(\theta_{z_o;\rho_{z_o}})}(z_o)\big|}
	\iint_{Q_{4\rho_{z_o}}^{(\theta_{z_o;\rho_{z_o}})}(z_o)\cap \boldsymbol  F(r_2,\eta\lambda)} 
	(G+|u|^{1+m} )\dxdt ,
\end{align}
for a constant  $c=c(n,m,\nu ,L)\ge1$. 
After choosing $\eta$ in such a way that  $c\, \eta^{2m}=\frac{1}{2}$,
we can re-absorb the resulting term $\frac12\lambda^{2m}$ into the left-hand side. Furthermore, for the second last term we apply H\"older's inequality and \eqref{<lambda-1} to estimate 
\begin{align*}
	\Bigg[\frac{1}{\big|Q_{4\rho_{z_o}}^{(\theta_{z_o;\rho_{z_o}})}(z_o)\big|} &
	\iint_{Q_{4\rho_{z_o}}^{(\theta_{z_o;\rho_{z_o}})}(z_o)\cap \boldsymbol E(r_2,\eta\lambda)} 
	\big|D\power{u}{m}\big|^{2q_o} \dxdt \Bigg]^{\frac{1}{q_o}-1} \\
	&\le
	\bigg[\biint_{Q_{4\rho_{z_o}}^{(\theta_{z_o;\rho_{z_o}})}(z_o)} 
	\big|D\power{u}{m}\big|^{2} \dxdt \bigg]^{1-q_o}
	\le
	c\lambda^{2m(1-q_o)}.
\end{align*}
Inserting this into \eqref{est-lambda2m} and observing from
\eqref{<lambda-1} that 
\begin{align*}
	\lambda^{2m}
	&
	\ge
	\frac{1}{c_o^{\frac{nd(1-m)}{2(1+m)}\big(n+2+\frac2m\big)}}
	\biint_{Q_{ c_o\rho_{z_o}}^{(\theta_{z_o;\rho_{z_o}})}(z_o)} 
	\big|D\power{u}{m}\big|^2 \dxdt
\end{align*}
holds true, we deduce that
\begin{align}\label{level-est}
	\iint_{Q_{ c_o\rho_{z_o}}^{(\theta_{z_o;\rho_{z_o}})}(z_o)} 
	\big|D\power{u}{m}\big|^2 \dxdt 
	&\le
    c
    \iint_{Q_{4\rho_{z_o}}^{(\theta_{z_o;\rho_{z_o}})}(z_o)\cap \boldsymbol E(r_2,\eta\lambda)} 
	\lambda^{2m(1-q_o)}\big|D\power{u}{m}\big|^{2q_o} \dxdt \nonumber\\
	&\quad +
	c
	\iint_{Q_{4\rho_{z_o}}^{(\theta_{z_o;\rho_{z_o}})}(z_o)\cap \boldsymbol F(r_2,\eta\lambda)} 
	(G+|u|^{1+m}) \dxdt
\end{align}
with $c=c(n,m,\nu ,L)>0$, where $c_o$ is given by Lemma~\ref{lem:vitali}.

\subsection{Covering argument}
We now observe that the upper-level set $\boldsymbol E(r_1,\lambda)$
with $\lambda>N\lambda_o$ can be covered by the family
$\mathcal
F\equiv\big\{Q_{4\rho_{z_o}}^{(\theta_{z_o;\rho_{z_o}})}(z_o):
z_o\in\boldsymbol E(r_1,\lambda) \big\}$ of parabolic cylinders such that
the estimate \eqref{level-est} holds on each cylinder. In view of the Vitali type Covering Lemma \ref{lem:vitali} we gain a countable subcollection
$$
	\Big\{Q_{4\rho_{z_i}}^{(\theta_{z_i;\rho_{z_i}})}(z_i)\Big\}_{i\in\N}
	\subset \mathcal F
$$
of pairwise disjoint parabolic cylinders with  
$$
	\boldsymbol E(r_1,\lambda)
	\subset 
	\bigcup_{i=1}^\infty Q_{c_o\rho_{z_i}}^{(\theta_{z_i;\rho_{z_i}})}(z_i)
	\subset
	Q_{r_2}.
$$
This implies that
\begin{align}
	\iint_{\boldsymbol E(r_1,\lambda)} 
	\big|D\power{u}{m}\big|^2 \dxdt
	&\le
	\sum_{i=1}^\infty
	\iint_{Q_{c_o\rho_{z_i}}^{(\theta_{z_i;\rho_{z_i}})}(z_i)} 
	\big|D\power{u}{m}\big|^2 \dxdt. \nonumber
\end{align}
Applying \eqref{level-est} and observing that the cylinders $Q_{4\rho_{z_i}}^{(\theta_{z_i;\rho_{z_i}})}(z_i)$ are pairwise disjoint we obtain
\begin{align*}
	\iint_{\boldsymbol E(r_1,\lambda)} 
	\big|D\power{u}{m}\big|^2 \dxdt
	&\le
	c\!\iint_{\boldsymbol E(r_2,\eta\lambda)} \!
	\lambda^{2m(1-q_o)}\big|D\power{u}{m}\big|^{2q_o} \dxdt \\\nonumber
	&\qquad+
	c\!\iint_{\boldsymbol F(r_2,\eta\lambda)} \!
	(G+|u|^{1+m}) \dxdt.
\end{align*}
To compensate for the different level sets appearing on both sides of the previous inequality, 
we use the fact that $|D\power um|^2\le\lambda^{2m}$ holds a.e. on
$\boldsymbol E(r_1,\eta\lambda)\setminus \boldsymbol E(r_1,\lambda)$,
which yields the bound
\begin{align*}
	\iint_{\boldsymbol E(r_1,\eta\lambda)\setminus \boldsymbol E(r_1,\lambda)} 
	\big|D\power{u}{m}\big|^2 \dxdt
	&\le
	\iint_{\boldsymbol E(r_2,\eta\lambda)} 
	\lambda^{2m(1-q_o)}\big|D\power{u}{m}\big|^{2q_o} \dxdt.
\end{align*}
We now add the last two inequalities to get
\begin{align*}
	\iint_{\boldsymbol E(r_1,\eta\lambda)} 
	\big|D\power{u}{m}\big|^2 \dxdt 
	&\le
	c\iint_{\boldsymbol E(r_2,\eta\lambda)} 
	\lambda^{2m(1-q_o)}\big|D\power{u}{m}\big|^{2q_o} \dxdt\\
	&\qquad +
	c\iint_{\boldsymbol F(r_2,\eta\lambda)} (G+|u|^{1+m}) \dxdt.
\end{align*}
Replacing  $\eta\lambda$ by $\lambda$ and recalling that $\eta<1$ depends only on $n,m,\nu$, and $L$, we eventually obtain the reverse H\"older type inequality
\begin{align}\label{pre-1}
	\iint_{\boldsymbol E(r_1,\lambda)} 
	\big|D\power{u}{m}\big|^2 \dxdt
	&\le
	c\iint_{\boldsymbol E(r_2,\lambda)} 
	\lambda^{2m(1-q_o)}\big|D\power{u}{m}\big|^{2q_o} \dxdt  \nonumber\\
	&\qquad+
	c \iint_{\boldsymbol F(r_2,\lambda)} (G+|u|^{1+m}) \dxdt
\end{align}
for any $\lambda\ge \lambda_1:= \eta N\lambda_o$.

\subsection{The desired gradient estimate}
For any $\ell\ge \lambda_1$ we first define 
the {\it truncation} of $|D\power{u}{m}|$ by
$$
	\big|D\power{u}{m}\big|_\ell
	:=
	\min\big\{\big|D\power{u}{m}\big|, \ell^m\big\},
$$
and for any $\rho\in(0,2R]$ the corresponding upper level set
$$
	\boldsymbol E_\ell(\rho,\lambda)
	:=
	\big\{z\in Q_\rho: \big|D\power{u}{m}\big|_\ell>\lambda^{m}\big\}.
$$
For all $\ell\ge\lambda_1$ we then calculate 
\begin{align}\label{est-1}
\iint_{Q_{r_1}}|D\power um|_\ell^{2-2q_o+\tau}|D\power um|^{2q_o}\dxdt
&\le
\lambda_1^{\tau m}\iint_{Q_{r_1}\setminus \boldsymbol E_\ell(r_1,\lambda_1)} |D\power um|_\ell^{2-2q_o}|D\power um|^{2q_o}\dxdt\nonumber\\
&\qquad+
\iint_{\boldsymbol E_\ell(r_1,\lambda_1)}|D\power um|_\ell^{2-2q_o+\tau}|D\power um|^{2q_o}\dxdt,
\end{align}
where $\tau\in(0,1)$ will be chosen later depending only on $n,m,\nu$, and $L$.
Since Fubini's theorem implies 
\begin{align}
&\iint_{\boldsymbol E_\ell(r_1,\lambda_1)}|D\power um|_\ell^{2-2q_o+\tau}|D\power um|^{2q_o}\dxdt\nonumber\\
&=
\lambda_1^{\tau m}\iint_{\boldsymbol E_\ell(r_1,\lambda_1)}|D\power um|_\ell^{2-2q_o}|D\power um|^{2q_o}\dxdt\nonumber\\
&\qquad+
\int_{\lambda_1}^\ell \tau m \lambda^{\tau m-1}\bigg[\iint_{\boldsymbol E_\ell(r_1,\lambda)}|D\power um|_\ell^{2-2q_o}|D\power um|^{2q_o}\dxdt\bigg]\d\lambda,\nonumber
\end{align}
we infer from \eqref{est-1} that
\begin{align}\label{est-1-1}
&\iint_{Q_{r_1}}|D\power um|_\ell^{2-2q_o+\tau}|D\power um|^{2q_o}\dxdt\nonumber\\
&\quad\le
\lambda_1^{\tau m}\iint_{Q_{r_1}}|D\power um|^{2}\dxdt
+
\int_{\lambda_1}^\ell\tau m\lambda^{\tau m-1}\bigg[\iint_{\boldsymbol E(r_1,\lambda)}|D\power um|^{2}\dxdt\bigg]\d\lambda.
\end{align}
Here, we used the facts $|D\power um|_\ell\le|D\power um|$ a.e. and
$\boldsymbol E_\ell(r_1,\lambda)=\boldsymbol E(r_1,\lambda)$ for
$\lambda<\ell$.
For the estimate of the last integral from the previous estimate we
employ \eqref{pre-1}. Furthermore, we again use $\boldsymbol
E_\ell(r_1,\lambda)=\boldsymbol E(r_1,\lambda)$ for $\lambda<\ell$
and Fubini's theorem. In this way, we obtain
\begin{align*}
&\int_{\lambda_1}^\ell\tau m\lambda^{\tau m-1}\bigg[\iint_{\boldsymbol E(r_1,\lambda)}|D\power um|^{2}\dxdt\bigg]\d\lambda\\
&\le
c\tau\iint_{\boldsymbol E_\ell(r_2,\lambda_1)}|D\power um|^{2q_o}
\bigg[ \int_{\lambda_1}^{|D\power um|_\ell^\frac1m} \lambda^{m(2-2q_o+\tau)-1}\d\lambda\bigg]\dxdt\\
&\qquad+
c\tau\iint_{\boldsymbol F(r_2,\lambda_1)}(G+|u|^{1+m}) \bigg[ \int_{\lambda_1}^{(G+|u|^{1+m})^\frac1{2m}} \lambda^{\tau m -1}\d\lambda\bigg]\dxdt\\
&\le
\frac{c\tau}{1-q_o}\iint_{Q_{r_2}}|D\power um|_\ell^{2-2q_o+\tau}|D\power um|^{2q_o}\dxdt\\
&\qquad+
c\iint_{Q_{r_2}}(G+|u|^{1+m})^{1+\frac\tau2}\dxdt.
\end{align*}
Inserting this into \eqref{est-1-1}, we deduce 
\begin{align}
&\iint_{Q_{r_1}}|D\power um|_\ell^{2-2q_o+\tau}|D\power um|^{2q_o}\dxdt\nonumber\\
&\qquad\le
\frac{c_o\tau}{1-q_o}\iint_{Q_{r_2}}|D\power um|_\ell^{2-2q_o+\tau}|D\power um|^{2q_o}\dxdt\nonumber\\
&\qquad\qquad+
\lambda_1^{\tau m}\iint_{Q_{r_2}}|D\power um|^2\dxdt
+
c\iint_{Q_{r_2}}(G+|u|^{1+m})^{1+\frac\tau2}\dxdt
\end{align}
for some $c_o=c_o(n,m,\nu,L)\ge1$.
The second last term can be estimated by means of
\begin{equation*}
  \lambda_1=(\eta N\lambda_o)\le
  c\bigg(\frac{4c_oR}{r_2-r_1}\bigg)^{\frac{d(n+2)(1+m)}{4m^2}}\lambda_o.
\end{equation*}
We now choose $\tau>0$ small enough to ensure
\begin{align*}
0<\tau\le\min\Big\{\tau_o,\gamma, \frac{4m}{1+m}+\frac4n-2\Big\}
\quad\mbox{where $\tau_o:=\frac{1-q_o}{2c_o}<1.$}
\end{align*}
For these values of $\tau$, the preceding estimate implies 
\begin{align}\label{before-iteration}\nonumber
&\iint_{Q_{r_1}}|D\power um|_\ell^{2-2q_o+\tau}|D\power um|^{2q_o}\dxdt\\
&\qquad\le
\frac12\iint_{Q_{r_2}}|D\power um|_\ell^{2-2q_o+\tau}|D\power um|^{2q_o}\dxdt\nonumber\\
&\qquad\qquad+
c\Big(\frac{R}{r_2-r_1}\Big)^{\frac{d(n+2)(1+m)}{4m}}\lambda_o^{\tau m}\iint_{Q_{2R}}|D\power um|^2\dxdt\nonumber\\
&\qquad\qquad+
c\iint_{Q_{2R}}(G+|u|^{1+m})^{1+\frac\tau2}\dxdt
\end{align}
for all $r_1$ and $r_2$ with $R\le r_1<r_2\le 2R$.
Here, we point out that the right-hand side is finite by the choice of
$\tau$, assumption~\eqref{def-G} and the property
$u\in L^{2m+\frac{2(1+m)}n}$ that follows from $u\in K_\psi$, cf. \eqref{integrability-Kpsi}.
Next, we employ Lemma~\ref{lem:tech} to eliminate the first term on
the right-hand side of \eqref{before-iteration}. Then we pass to the
limit $\ell\rightarrow \infty$ and apply Fatou's Lemma to obtain
\begin{align*}
\iint_{Q_{R}}|D\power um|^{2+\tau}\dxdt
\le
c\lambda_o^{\tau m}\iint_{Q_{2R}}|D\power um|^2\dxdt
+
c\iint_{Q_{2R}}(G+|u|^{1+m})^{1+\frac\tau2}\dxdt.
\end{align*}
In order to estimate $\lambda_o$, we recall \eqref{def:lambda_o},
adopt the energy bound~\eqref{energy:a=0} from Lemma~\ref{lem:energy} with $\theta=1$ and use Young's inequality to deduce
\begin{align*}
\lambda_o
\le
c\Bigg[ \biint_{Q_{8R}} \bigg(\frac{|u|^{1+m}}{R^\frac{1+m}m}+G+1\bigg)\dxdt\Bigg]^\frac{d}{2m}.
\end{align*}
We then combine the two preceding estimates, which leads to
\begin{align*}
&\biint_{Q_{R}}|D\power um|^{2+\tau}\dxdt\\
&\qquad\le
c\Bigg[ \biint_{Q_{8R}} \bigg(\frac{|u|^{1+m}}{R^\frac{1+m}m}+G+1\bigg)\dxdt\Bigg]^\frac{d\tau }{2} \biint_{Q_{2R}}|D\power um|^2\dxdt\\
&\qquad\qquad+
c\,\biint_{Q_{2R}}(G+|u|^{1+m})^{1+\frac\tau2}\dxdt.
\end{align*}
To obtain the desired estimate \eqref{desired-est-1}, we cover
$Q_R(z)$ by finitely many cylinders of radius $\frac{R}8$ and then apply the
previous estimate on each of the smaller cylinders. After summing up, we
deduce the asserted estimate~\eqref{desired-est-1}. This completes the
proof of Theorem~\ref{thm:main}.
\end{proof}


\begin{thebibliography}{99}

\bibitem{AdimurthiByun:2018}
  K.~Adimurthi, S.~Byun.
  \newblock Boundary higher integrability for very weak solutions of
  quasilinear parabolic equations.
  \newblock \emph{J. Math. Pures Appl. (9)}, 121 (2019), 244--285. 
  
\bibitem{AltLuckhaus}
\newblock H. Alt and S. Luckhaus.
\newblock Quasilinear elliptic-parabolic differential equations.
\newblock {\em Math. Z.}, 183 (1983), no.~3, 311--341.

\bibitem{Boegelein:1}
\newblock V.~B\"ogelein.
\newblock Higher integrability for weak solutions of higher order degenerate parabolic systems. 
\newblock {\em Ann.~Acad.~Sci.~Fenn.~Math.}, 33 (2008), no.~2, 387--412.

\bibitem{Boegelein:2}
\newblock V.~B\"ogelein.
\newblock Very weak solutions of higher-order degenerate parabolic
systems.
\newblock \emph{Adv. Differential Equations} 14 (2009), no. 1-2, 121--200.


\bibitem{Boegelein-Duzaar:1}
\newblock V.~B\"ogelein and F.~Duzaar.
\newblock Higher integrability for parabolic systems with non-standard growth and degenerate diffusions. 
\newblock {\em Publ.~Mat.}, 55 (2011), no.~1, 201--250.


\bibitem{BDKS:2018}
\newblock V.~B\"ogelein, F.~Duzaar, J.~Kinnunen, and C.~Scheven.
\newblock Higher integrability for doubly nonlinear parabolic systems.
To appear in \emph{J. Math. Pures Appl. (9)}.


\bibitem{Bogelein-Duzaar-Korte-Scheven}
\newblock V.~B\"ogelein, F.~Duzaar, R.~Korte, C.~Scheven.
\newblock The higher integrability of weak solutions of porous medium systems. 
\newblock {\em Adv. Nonlinear Anal.}, 8 (2019), 1004--1034.

\bibitem{BDS-singular-higher-int}
\newblock V.~B\"ogelein, F.~Duzaar, and C.~Scheven.
\newblock Higher integrability for the singular porous medium system.
\newblock To appear in \emph{J. Reine Angew. Math. (Crelle's Journal)},
doi:10.1515/crelle-2019-0038.

\bibitem{BLS:MathAnn05}
\newblock V.~B\"ogelein, T.~Lukkari, and C.~Scheven.
\newblock The obstacle problem for the porous medium equation.
\newblock {\em Math.~Ann.}, 363 (2015), no.~1, 455--499.


\bibitem{Boegelein-Parviainen}
\newblock V.~B\"ogelein and M.~Parviainen.
\newblock Self-improving property of nonlinear higher order parabolic systems near the boundary. 
\newblock  {\em NoDEA Nonlinear Differential Equations Appl.}, 17 (2010), no.~1, 21--54.

\bibitem{Boegelein-Scheven-obs}
\newblock V.~B\"ogelein and C. Scheven.
\newblock Higher integrability in parabolic obstacle problems.
\newblock {\em Forum Math.}, 24 (2012), no.~5, 931--972.


\bibitem{Cho-Scheven}
  \newblock Y.~Cho, C.~Scheven.
  \newblock The self-improving property of higher integrability in the
  obstacle problem for the porous medium equation.
  \newblock \emph{NoDEA Nonlinear Differential Equations Appl.} 26
  (2019), no. 5, Art. 37, 44p. 

\bibitem{DiBe}
E.~DiBenedetto.
\newblock {\em Degenerate parabolic equations}.
\newblock Springer-Verlag, Universitytext xv, 387, New York, NY, 1993.

\bibitem{DBGV-book} 
E. DiBenedetto, U. Gianazza, and V. Vespri. 
{\it Harnack's inequality for degenerate and singular parabolic equations}.
Springer Monographs in Mathematics, 2011.

\bibitem{Eleuteri:2007}
M.~Eleuteri.
\newblock Regularity results for a class of obstacle problems.
\newblock {\em Appl. Math.} 52 (2007), no. 2,137--170.

\bibitem{EleuteriHabermann:2008}
M.~Eleuteri and J.~Habermann.
\newblock Regularity results for a class of obstacle problems under nonstandard growth conditions.
\newblock {\em J. Math. Anal. Appl.}, 344 (2008), no. 2, 1120--1142.


\bibitem{Federer}
\newblock H.~Federer.
\newblock {\em Geometric measure theory.} 
\newblock Die Grundlehren der mathematischen Wissenschaften, Band 153, Springer-Verlag, New York, 1969.

\bibitem{Gehring}
\newblock F.~W.~Gehring.
\newblock The $L^p$-integrability of the partial derivatives of a quasiconformal mapping.
\newblock {\em Acta Math.}, 130 (1973), 265--277.

\bibitem{Gianazza-Schwarzacher}
\newblock U.~Gianazza and S.~Schwarzacher.
\newblock Self-improving property of degenerate parabolic equations of porous medium-type.
\newblock \emph{Amer. J. Math.} 141 (2019), no. 2, 399--446.


\bibitem{Gianazza-Schwarzacher:m<1}
\newblock U.~Gianazza and S.~Schwarzacher.
\newblock Self-improving property of the fast diffusion equation.
\newblock \emph{J. Funct. Anal.} 277 (2019), no. 12, 108291.


\bibitem{Giaquinta:book}
\newblock M.~Giaquinta.
\newblock Multiple Integrals in the Calculus of Variations and Nonlinear Elliptic Systems.
\newblock Princeton University Press, Princeton, 1983.

\bibitem{Giaquinta-Struwe}
\newblock M.~Giaquinta and M.~Struwe.
\newblock On the partial regularity of weak solutions of nonlinear parabolic systems.
\newblock {\em Math.~Z.}, 179 (1982), no.~4, 437--451.

\bibitem{Giusti:book}
\newblock E.~Giusti.
\newblock {\em Direct Methods in the Calculus of Variations}.
\newblock World Scientific Publishing Company, Tuck Link, Singapore, 2003.

\bibitem{Kilpelainen-Koskela}
T.~Kilpel{\"a}inen and P.~Koskela.
\newblock Global integrability of the gradients of solutions to partial
  differential equations.
\newblock {\em Nonlinear Anal.}, 23 (1994), no.~7, 899--909.

\bibitem{Kinnunen-Lewis:1}
\newblock J.~Kinnunen and J.~L.~Lewis.
\newblock Higher integrability for parabolic systems of $p$-Laplacian type.
\newblock {\em Duke Math.~J.}, 102 (2000), no.~2, 253-271.

\bibitem{QifanLi}
  \newblock Q.~Li.
  \newblock Higher integrability for obstacle problem related to the
  singular porous medium equation.
  \newblock Preprint.

\bibitem{Meyers-Elcrat}
\newblock N.~G.~Meyers and A.~Elcrat.
\newblock Some results on regularity for solutions of non-linear elliptic
systems and quasi-regular functions.
\newblock {\em Duke Math. J.}, 42 (1975), 121--136.

\bibitem{MSSS}
  K.~Moring, C.~Scheven, S.~Schwarzacher, and T.~Singer.
  \newblock Global higher integrability of weak solutions
  of porous medium systems.
  \newblock \emph{Comm. Pure Appl. Anal. }19 (2020), no. 3, 1697--1745.
  
\bibitem{Parviainen}
\newblock M.~Parviainen.
\newblock Global gradient estimates for degenerate parabolic equations in nonsmooth domains. 
\newblock{\em Ann.~Mat.~Pura Appl.} (4), 188 (2009), no.~2, 333--358.

\bibitem{Parviainen-singular}
  M.~Parviainen.
  \newblock Reverse H\"older inequalities for singular parabolic
  equations near the boundary.
  \newblock \emph{J. Differential Equations} 246(2):512--540, 2009.

\bibitem{Saari-Schwarzacher}
  O. Saari and S. Schwarzacher,
  \newblock A reverse H\"older inequality for the gradient of solutions to
  Trudinger's equation.
  \newblock Preprint.


\bibitem{Schwarzacher}
\newblock S.~Schwarzacher.
\newblock H\"older-Zygmund estimates for degenerate parabolic systems.
\newblock {\em J. Differential Equations}, 256 (2014), no.~7, 2423--2448.


\end{thebibliography}
\end{document}